\documentclass[preprint]{elsarticle}
\makeatletter 
\def\ps@pprintTitle{%
 \let\@oddhead\@empty
 \let\@evenhead\@empty
 \def\@oddfoot{\centerline{\thepage}}%
 \let\@evenfoot\@oddfoot}
\makeatother
\usepackage[utf8]{inputenc}
\usepackage[english]{babel}
\usepackage{amssymb}
\usepackage{amsmath}
\usepackage{amsthm}
\usepackage{commath}
\usepackage{hyperref}
\usepackage{lmodern}
\usepackage{url}

\newtheorem{theorem}{Theorem}[section]

\newtheorem{lemma}[theorem]{Lemma}
\newtheorem{proposition}[theorem]{Proposition}

\numberwithin{equation}{section}

\theoremstyle{definition}
\newtheorem*{definition}{Definition}

\theoremstyle{remark}
\newtheorem*{remark}{Remark}

\DeclareMathOperator{\sech}{sech}

\DeclareMathOperator{\codim}{codim}
\DeclareMathOperator{\coker}{coker}

\journal{Journal of Differential Equations}

\begin{document}

\begin{frontmatter}

\title{Solutions of Liouville equations with non-trivial profile in dimensions 2 and 4}

\author{Roberto Albesiano}
\ead{ralbesiano@math.stonybrook.edu}
\address{Mathematics Department, Stony Brook University, Stony Brook NY, 11794-3651, USA}

\begin{abstract}
  We prove the existence of a family of non-trivial solutions of the Liouville equation in dimensions two and four with infinite volume. These solutions are perturbations of a finite-volume solution of the same equation in one dimension less. In particular, they are periodic in one variable and decay linearly to $-\infty$ in the other variables. In dimension two, we also prove that the periods are arbitrarily close to $\pi k, k \in \mathbb{N}$ (from the positive side). The main tool we employ is bifurcation theory in weighted H\"older spaces.
\end{abstract}

\end{frontmatter}

\section{Introduction}
Liouville equations are a class of elliptic nonlinear partial differential equations of the form
\begin{equation}\label{eq:Liouville}
(-\Delta)^n \varphi(x) = \text{e}^{\varphi(x)}, \quad x \in \mathbb{R}^{2n},
\end{equation}
for $n \in \mathbb{N}$.\footnote{We will deal only with spaces of even dimension. For the odd-dimensional case, which is much more difficult as it involves the fractional Laplacian, we refer to \cite{Hyder2017}.} This family of equations plays a fundamental role in many problems of Conformal Geometry and Mathematical Physics, governing the transformation laws for some curvatures. For example, the 2-dimensional equation (which is often called \emph{classical Liouville equation}) provides the structure of metrics with constant Gaussian curvature. Indeed, let $g$ be a metric on a surface $M$: if we conformally rescale the metric as $\hat{g}_{ij} = \text{e}^{2\varphi} g_{ij}$, for some smooth function $\varphi$ on $M$, then
\[
R_{\hat{g}} = \text{e}^{-2\varphi}(R_g - 2\Delta f),
\]
where $R_{\hat{g}}$ and $R_g$ denote, respectively, the Gaussian curvatures of $\hat{g}$ and $g$. Specifically, if $g$ is an Euclidean metric, then we obtain the 2D Liouville equation
\[
\Delta \varphi(u,v) + 2K \text{e}^{\varphi(u,v)} = 0.
\]

In Mathematical Physics, Liouville equations appear for example in the description of mean field vorticity in steady flows (\cite{Caglioti1995}, \cite{Chanillo1994}), Chern-Simons vortices in superconductivity or Electroweak theory (\cite{Tarantello2008}, \cite{YYang2001}). Moreover, they also arise naturally when dealing with functional determinants, which play an essential role in modern Quantum Physics and String theory \cite{Osgood1988}. The 2-dimensional Liouville equation was also taken as an example by David Hilbert in the formulation of the ``nineteenth problem'' \cite{Hilbert1902}.

Liouville equations, therefore, have been extensively studied in the past years, and many authors tried to find their non-trivial solutions. For the classical Liouville equation, for instance, solutions in $\mathbb{R}^2$ with finite \emph{volume} $V := \int \exp(u)$ have been completely classified:

\begin{theorem}[\textsc{Chen - Li} \cite{Chen1991}]
  Every solution of
  \[
  \begin{cases}
    \Delta u + \text{\emph{e}}^u = 0 \quad \text{in } \mathbb{R}^2 \\
    \int_{\mathbb{R}^2} \text{\emph{e}}^{u(x)} \dif x < +\infty
  \end{cases}
  \]
  is radially symmetric with respect to some point in $\mathbb{R}^2$. Specifically, it assumes the form
  \[
  u(x) = \phi_{\lambda,x^0}(x) = \frac{\log (32 \lambda^2)}{(4 + \lambda^2 |x-x^0|^2)^2}
  \]
  for some $\lambda > 0$ and $x^0 \in \mathbb{R}^2$.
\end{theorem}

The infinte-volume case, though, is still quite unexplored. The first main result of this paper is the following.

\begin{theorem}\label{thm:2Dnon-trivialSolution}
  The classical Liouville equation $\Delta u + \text{\emph{e}}^u = 0$ in $\mathbb{R}^2$ admits a class of non-trivial solutions $u \in C^{2,\alpha}(\mathbb{R}^2)$ with infinite volume. These solutions tend symmetrically to $-\infty$ along the $x$ direction and are periodic of period arbitrarily close to $\pi k$ for $k \in \mathbb{N}_{>0}$ in the $y$ direction.

  More precisely, for every $k \in \mathbb{N}_{>0}$, we have a one-dimensional family of non-trivial solutions which are perturbations of $\log(2\sech(x))$ along the $y$ variable. This perturbation, at first order, is given by $\sech(x) \cos(\pi k y)$.
\end{theorem}

As one can tell from the statement, the result relies on the fact that there exists a class of solutions in dimension 1, that can be used to build a ``trivial'' solution with infinite volume in dimension 2 (i.e.~a solution that does not actually depend on all variables). Specifically, the family of solutions in dimension 1 is
\[
\log\left[c_1 - c_1\tanh^2\left(\frac{1}{2}\sqrt{2c_1\left(c_2+x\right)^2}\right)\right], \quad c_1 \geq 0, \quad c_2 \in \mathbb{R}.
\]
Extending them trivially along $y$ gives the class of trivial solutions. Notice that Theorem \ref{thm:2Dnon-trivialSolution} can actually be generalized to perturbations of any element in this class simply by changing variables.

After we have a trivial solution, a bifurcation theorem comes into play: the \emph{Bifurcation from the simple eigenvalue Theorem} \ref{thm:1Dbifurcation}, in this case. The main difficulty here is to prove that the linearized operator is Fredholm of index 0, which is accomplished via some growth arguments (Lemma \ref{lemma:candidateLambda}) and weighted elliptic regularity theory (Lemma \ref{lemma:linearizedFredholm}). In particular, a weighted version of Schauder's estimates (Theorem \ref{thm:weightedSchauder}) is necessary for the proof. We also find that the bifurcation is supercritical and, as a byproduct, we correct the formula for non-transcritical bifurcations (Proposition \ref{prop:bifShape}). Section \ref{sec:2D} deals with the classical Liouville equation.

\begin{remark}
  In real dimension 2, one might be able to find non-trivial solutions of \eqref{eq:Liouville} in terms of meromophic functions. For instance, on a simply connected domain $\Omega \subseteq \mathbb{C}$, the general solution is given by
  \[
  u(z,\bar{z}) = \log \left( 4 \frac{|\partial f(z) / \partial z|}{(1+K|f(z)|^2)^2} \right),
  \]
  where $f$ is any meromorphic function such that $\pd{f}{z}(z) \neq 0$ for all $z \in \Omega$ and $f$ has at most simple poles in $\Omega$ (see \cite{Henrici1974} -- see also \cite{Galvez2012} for a classification of solutions with finite volume in the upper half-plane). Notice though that this only applies to real dimension 2, and does not extend to higher dimension, which instead we are able to treat here with the above bifurcation method.
\end{remark}

\subsection{Higher-dimensional Liouville equations}
The role that Liouville equations have in Conformal Geometry does not hold only for constant Gaussian curvature surfaces. Indeed, more generally, the \emph{generalized Liouville equation} \eqref{eq:Liouville} plays also a fundamental part in constant \emph{$Q$-curvature} manifolds.

In 1985 Thomas P.~Branson introduced the concept of $Q$-curvature \cite{Branson1985}, a quantity that turned out to be very important in many contexts and that can be regarded as a generalization of the Gaussian curvature. For example, $Q$-curvature appears naturally when studying the \emph{functional determinant} of conformally covariant operators, which is a fundamental concept of both Functional Analysis and Theoretical Physics\footnote{Given an operator $A$ with spectrum $\{\lambda_j\}_j$, one can formally define its \emph{determinant} as $\prod_j \lambda_j$. This is divergent, in general, so one should perform some sort of ``regularization'' of the definition. Define then the \emph{Zeta function} as
\[
\zeta(s) := \sum_j \lambda_j^{-s} = \sum_j \text{e}^{-s \log \lambda_j}.
\]
One can show by means of Weyl's asymptotic law (see for example \cite[Chapter 11]{Strauss2009}) that this defines an analytic function for $\Re(s) > n/2$ if $A$ is the Laplace-Beltrami opeator. Moreover, one can meromorphically extend $\zeta$ so that it becomes regular at $s=0$ \cite{Ray1971}. Taking the derivative, one has $\zeta'(0) := - \sum_j \log \lambda_j = -\log \det A$, so that $\det A := \exp(-\zeta'(0))$. For more details we refer to \cite{Osgood1988}, \cite{Chang2008}, \cite{Gursky2012} and the references therein.
}. For example, on a four-manifold, given a conformally covariant operator $A_g$ (like the conformal Laplacian or the Paneitz operator \cite{Paneitz2008}) and a conformal factor $w$, one has
\[
\log \frac{\det A_{\hat{g}}}{\det A_g} = \gamma_1(A) F_1[w] + \gamma_2(A) F_2[w] + \gamma_3(A) F_3[w],
\]
where $\gamma_1(A)$, $\gamma_2(A)$ and $\gamma_3(A)$ are real numbers \cite{Branson1991}. In particular, $\hat{g} = \text{e}^{2w}g$ is a critical point of $F_2$ if and only if the $Q$-curvature corresponding to $\hat{g}$ is constant (see \cite{Esposito2019}, \cite{Gursky2012} and the references therein).

$Q$-curvature appears also as the 0-th order term of the GJMS-operator in the ambient metric construction \cite{Fefferman2013} and can be related to the Poincaré metric in one higher dimension via an ``holographic formula'' \cite{Graham2007}. GJMS-operators, in turn, play an important role in Physics, as their definition extends to Lorentzian manifolds: they are generalizations of the Yamabe operator and the conformally covariant powers of the wave operator on Minkowski's space \cite{Juhl2009}. Moreover, the integral of the $Q$-curvature satisfies the so-called Chern-Gauss-Bonnet formula \cite{Juhl2009}, which links the integral of some function of the $Q$-curvature to the Euler characteristic of the manifold (as the Gauss-Bonnet formula did with the Gaussian curvature). In $\mathbb{R}^4$, the integral of $Q$-curvature can tell us whether a metric is normal and, in that case, is strictly related to the behavior of the isoperimetric ratios \cite{Chang2000}.

In what follows, we will deal only with the 2- and 4-dimensional cases. A more complete introduction to 2, 4 and higher dimensional $Q$-curvature can be found in \cite{Chang2008}. See \cite{Branson1994} for a generic definition of $Q$-curvature and \cite{Juhl2009} for explicit formulas.

In dimension 2 the $Q$-curvature is essentially the usual Gaussian curvature. As we have already seen, this leads to the classical Lioville equation. In dimension 4, this quantity starts to become more interesting.
\begin{definition}
Let $(M,g)$ be a 4-dimensional Riemannian manifold. Let $\text{Ric}_g$ be its Ricci curvature, $R_g$ its scalar curvature and $\Delta_g$ its Laplace-Beltrami operator. The \emph{$Q$-curvature} of $M$ is defined as
\[
Q_g := -\frac{1}{12}\left( \Delta_g R_g - R_g^2 + 3|\text{Ric}_g|^2 \right).
\]
\end{definition}
Conformally rescaling the metric, $\hat{g}_{ij} = \text{e}^{2\varphi} g_{ij}$ for some smooth function $\varphi$ on $M$, then the $Q$-curvature transforms as follows
\begin{equation} \label{eq:QconformalTransformation}
P_g \varphi + 2Q_g = 2Q_{\hat{g}} \text{e}^{4\varphi}
\end{equation}
(see for example \cite[Chapter 4]{Chang2004}), where $P_g$ is the \emph{Paneitz operator} \cite{Paneitz2008}
\[
P_g \varphi := \Delta_g^2 \varphi + \text{div}_g \left(\frac{2}{3} R_g g - 2 \text{Ric}_g\right) \dif \varphi.
\]

Observe that, if we take $M = \mathbb{R}^4$ and $g$ equal to the standard Euclidean metric and consider $\hat{g}$ conformal to $g$ and such that $Q_{\hat{g}} \equiv \bar{Q} \in \mathbb{R}$, then equation \eqref{eq:QconformalTransformation} becomes
\[
\Delta_g^2 \varphi = 2\bar{Q} \text{e}^{4\varphi}.
\]
Setting $u := 4\varphi$ and $\bar{Q}=\frac{1}{8}$ and taking into account that the Laplace-Beltrami operator in $\mathbb{R}^4$ endowed with the Euclidean metric is the standard Laplacian, we finally end up with the 4-dimensional Liouville equation
\[
\Delta^2 u(x) = \text{e}^{u(x)}, \quad x \in \mathbb{R}^4.
\]

As for dimension 2, a complete classification is known for solutions of the 4-dimensional Liouville equation that have finite volume:

\begin{theorem}[\textsc{Lin} \cite{Lin1998}]
  Suppose that $u$ is a solution of
  \[
  \begin{cases}
    \Delta^2 u = 6 \text{\emph{e}}^{4u} \quad \text{in } \mathbb{R}^4 \\
    \text{\emph{e}}^{4u} \in L^1(\mathbb{R}^4)
  \end{cases},
  \]
  then the following statements hold.
  \begin{enumerate}
    \item After an orthorgonal transformation, $u(x)$ can be represented by
    \[
    \begin{split}
      u(x) &= \frac{3}{4\pi^2} \int_{\mathbb{R}^4} \log\left(\frac{|y|}{|x-y|}\right) \text{\emph{e}}^{4u(y)} \dif y - \sum_{j=1}^4 a_j(x_j - x_j^0)^2 + c_0 \\
           &= -\sum_{j=1}^4 a_j(x_j - x_j^0)^2 - \alpha \log|x| + c_0 + \text{O}(|x|^{-\tau})
    \end{split}
    \]
    for some $\tau>0$ and for large $|x|$, where $a_j \geq 0$ and $c_0$ are constants, and $x^0 = (x_1^0,x_2^0,x_3^0,x_4^0) \in \mathbb{R}^4$. Moreover, if $a_i \neq 0$ for all $i$, then $u$ is symmetric with respect to the hyperplane $\{x \mid x_i = x_i^0 \}$. If $a_1=a_2=a_3=a_4 \neq 0$, then $u$ is radially symmetric with respect to $x^0$.
    \item Let
    \[
    \alpha = \frac{3}{4\pi^2} \int_{\mathbb{R}^4} \text{\emph{e}}^{4u(y)} \dif y,
    \]
    then $\alpha \leq 2$. Moreover, if $\alpha = 2$, then
    \[
    u(x) = \log \frac{2\lambda}{1 + \lambda^2|x-x_0|^2}.
    \]
  \end{enumerate}
\end{theorem}

Non-standard solutions with finite volume in $\mathbb{R}^4$ have been constructed explicitly also in \cite{Wei2008}.

In Section \ref{sec:4D} we will show that there exists a nontrivial solution with infinite volume also in dimension 4. Specifically, we will prove the following theorem.

\begin{theorem}\label{thm:4Dnon-trivialSolution}
  The Liouville equation $\Delta^2 u = \text{\emph{e}}^u$ in $\mathbb{R}^4$ admits a non-trivial solution $u \in C^{4,\alpha}(\mathbb{R}^4)$ with infinite volume. In particular, this solution is radial in the first three coordinates $(x_1,x_2,x_3)$, decaying to $-\infty$ as $O(|(x_1,x_2,x_3)|)$ (i.e.~at most linearly), and is periodic in the last coordinate $x_4$.
\end{theorem}

Again, the philosophy is: proving the existence of a finite-volume solution in dimension 3, extending that solution trivially to $\mathbb{R}^4$ and then using a bifurcation theorem to prove that there is a nontrivial solution which is the perturbation of the trivial one. In this case, the existence of a trivial solution already (i.e.~a solution of $\Delta^2 u = \text{e}^u$ in $\mathbb{R}^3$) is not completely obvious and could be obtained only implicitly (Subsection \ref{subsec:trivial}). As for the non-trivial solution, the bifurcation theorem we employed is the more sophisticated \emph{Krasnosel'skii's Index Theorem} \ref{thm:krasnoselskii}. The reason that prevents us from using simpler bifurcation results as in the case of dimension 2 is that all the functions involved in higher dimension cannot be expressed explicitly. Hence, even if Theorem \ref{thm:2Dnon-trivialSolution} and Theorem \ref{thm:4Dnon-trivialSolution} look quite similar, the two proofs will be notably different, with the second having to resort to more abstract machinery, such as degree theory.

\subsection{Some perspectives}
The argument exposed in this work seems far from being the only way of finding nontrivial solutions of the Liouville equation. For instance, in four dimensions, one could try to follow the same procedure with a trivial solution in $\mathbb{R}^2$ and two parameters\footnote{Observe that choosing to have a trivial solution in dimension 1 leads nowhere. Indeed, we would be looking for solutions of the ODE
\[
u^{(4)}(x) = \text{e}^{u(x)}, \quad \forall x \in \mathbb{R}.
\]
Integrating this equation, one immediately finds
\[
u^{(3)}(y) - u^{(3)}(x) = \int_x^y \text{e}^{u(s)} \dif s > 0, \quad \forall x,y \in \mathbb{R}, x \neq y,
\]
which means that $\lim_{x \rightarrow -\infty} u^{(3)}(x) < \lim_{x \rightarrow +\infty} u^{(3)}(x)$. But this is a contradiction: as we are asking $u$ to have finite volume we need that $u$ goes to $-\infty$ both at $-\infty$ and $+\infty$, and so $\lim_{x \rightarrow -\infty} u^{(3)}(x) \geq 0$ and $\lim_{x \rightarrow +\infty} u^{(3)}(x) \leq 0$.}, but even this might not exhaust all possibilities. Another option, indeed, might be to bifurcate radially. Moreover, similarly to what happens for Delaunay surfaces \cite{Delaunay1841} (whose equations are somehow analogous to the ones we are considering), we might be able to connect the non-trivial solutions of Theorem \ref{thm:4Dnon-trivialSolution} to the radial solution by some global bifurcation theorem, maybe following the ideas of \cite{Dancer2001}.

Another aspect which is worth considering is, of course, going to higher dimensions. As the methods employed in Section \ref{sec:4D} are quite general, it seems reasonable to believe that the same argument can be generalized to any even dimension. Similarly to dimensions 2 and 4, the finite-volume case has already been classified in the works of L.~Martinazzi \cite{Martinazzi2008}, \cite{Martinazzi2008b} and \cite{Martinazzi2014} (with A.~Hyder). The existence of large volume solutions has been studied by the same author also in \cite{Martinazzi2013}.

Finally, other lines of research aimed at finding more general non-trivial solutions with infinite volume are possible. For instance, we might be able to ``glue'' the oscillating solutions obtained from the bifurcation into more complex solutions, in a manner similar to the one used to construct Delaunay $k$-noids starting from Delaunay unduloids and nodoids. In this way we could then obtain non-trivial solutions with infinite volume that are not a direct result of a bifurcation from a cylindrical solution. For an example of this phenomenon in a PDE context, see \cite{Malchiodi2009}.

\vspace*{10pt}
\vbox{
\begin{center}
  \textbf{Acknowledgments}
\end{center}
\vspace*{-5pt}
The author is grateful to Ali Hyder for providing ideas and details for Subsection \ref{subsec:trivial}, and to Andrea Malchiodi for encouraging the study of this problem and for the many useful discussions about it.}

\section{Preliminaries}
We list here some known theorems used in the proofs of the main results of this article.

\subsection{Bifurcation theory}
The results in this section are well known and can be found for instance in \cite{AmbrosettiProdi} and \cite{Kielhofer}. The only new contribution is Subsection \ref{subsubsec:bifShape}, in which we find a correct version of the formula discriminating between subcritical and supercritical bifurcations.
\begin{definition}
Let $X$ and $Y$ be Banach spaces, and $F \in C^2(\mathbb{R} \times X,Y)$ be such that $F(\lambda,0)=0$ for all $\lambda \in \mathbb{R}$. We say that $\lambda^*$ is a \emph{bifurcation point} for $F$ (from the trivial solution $u \equiv 0$) if there is a sequence of solutions $(\lambda_n, u_n)_{n \in \mathbb{N}} \subset \mathbb{R} \times X$, with $u_n \neq 0$ for each $n \in \mathbb{N}$, that converges to $(\lambda^*,0)$.
\end{definition}

A direct consequence of the Implicit Function Theorem is the following necessary condition for bifurcation.

\begin{proposition}\label{prop:bifNecessaryCondition}
  A necessary condition for $\lambda^*$ to be a bifurcation point for $F$ is that $F_u(\lambda^*,0) := \pd{F}{u}(\lambda^*,0)$ is not invertible.
\end{proposition}

Observe that, in order to prove non-uniqueness, it suffices to find a bifurcation point $\lambda^*$. In the 2D case, this will be achieved by means of the vary basic \emph{Bifurcation from the simple eigenvalue Theorem} \cite[Theorem 5.4.1]{AmbrosettiProdi}:

\begin{theorem}\label{thm:1Dbifurcation}
  Let $X$ and $Y$ be Banach spaces, $F \in C^2(\mathbb{R} \times X, Y)$ be such that $F(\lambda, 0) = 0$ for all $\lambda \in \mathbb{R}$. Let $\lambda^*$ be such that $L = F_u(\lambda^*, 0) \equiv \pd{F}{u}(\lambda^*, 0)$ has one-dimensional kernel $V = \{tu^* \mid t \in \mathbb{R} \}$ and closed range $R$ with codimension 1. Letting $M := F_{u \lambda}(\lambda^*,0)$, assume moreover that $M[u^*] \notin R$. Then $\lambda^*$ is a bifurcation point for $F$. In addition, the set of non-trivial solutions of $F=0$ is, near $(\lambda^*, 0)$, a unique $C^1$ cartesian curve with parametric representation on $V$.
\end{theorem}

Unfortunately, this very simple result has no application in the 4D case, because it does not seem possible to have explicit solutions of the differential equations involved. The result will be proved anyway using the more sophisticated \emph{Krasnosel'skii's Index Theorem} \cite[Theorem 56.2]{Krasnoselskii1984}:

\begin{theorem}\label{thm:krasnoselskii}
  Let $A$ be a completely continuous operator and assume that $\lambda^*$ is a point where the index of $u - A(\lambda,u)$ changes. Then $\lambda^*$ is a bifurcation point for equation $u = A(\lambda,u)$.
\end{theorem}

\subsubsection{Shape of bifurcation}\label{subsubsec:bifShape}
In the relatively simple 2D case, we can know something about the shape of the bifurcating branch from the explicit computations on the solutions of the linearized problem.

\begin{definition}
  Let $\lambda^*$ be as in Theorem \ref{thm:1Dbifurcation}. We say that the bifurcation is \emph{supercritical} (\emph{subcritical}) if $\lambda \geq \lambda^*$ ($\lambda \leq \lambda^*$) along the bifurcating branch, in some ball around $(\lambda^*,0)$. If, along the bifurcating solution, $\lambda$ assumes values both bigger and smaller than $\lambda^*$ in all balls around $(\lambda^*,0)$, then the bifurcation is said to be \emph{transcritical}.
\end{definition}

\begin{proposition}\label{prop:bifShape}
  In the setting of Theorem \ref{thm:1Dbifurcation}, suppose that $F$ is $C^\infty$. Then
  \begin{enumerate}
    \item If $\langle \psi, F_{uu}(\lambda^*,0)[u^*,u^*] \rangle \neq 0$, then the bifurcation is transcritical.
    \item If instead $\langle \psi, F_{uu}(\lambda^*,0)[u^*,u^*] \rangle = 0$, let $u_2 \in L^{-1}\left(-\frac{1}{2} F_{uu}(\lambda^*,0)[u^*,u^*]\right)$ and consider
    \[
    \lambda_2 = -\frac{\langle \psi, F_{uu}(\lambda^*,0)[u^*,u_2] \rangle + \frac{1}{6} \langle \psi, F_{uuu}(\lambda^*,0)[u^*,u^*,u^*] \rangle}{\langle \psi, F_{\lambda,u}(\lambda^*,0)[u^*]\rangle}.
    \]
    If $\lambda_2>0$, then the bifurcation is supercritical; while, if $\lambda_2<0$, then the bifurcation is subcritical.
  \end{enumerate}
\end{proposition}

\begin{remark}
  The formula for $\lambda_2$ that can be found in many books, like \cite[(I.6.11)]{Kielhofer} or \cite[(5.4.7)]{AmbrosettiProdi}, is not correct because it misses the term containing $u_2$. Observe that in general $u_2$ is different from 0 because of \eqref{eq:u2condition}.
\end{remark}

\begin{proof}
\begin{enumerate}
  \item This is well-known and can be found for instance in \cite[5.4.3.iv]{AmbrosettiProdi}. We summarize the proof here for convenience.

  We compute the first terms of the Taylor expansion of $F$ centered in $(\lambda^*,0)$: Let $\lambda_0 := \lambda^*$ and $u_1 := u^*$, then
  \[
  \begin{split}
  0 &= F(\lambda(t),u(t)) = F(\lambda_0 + \lambda_1 t + \lambda_2 t^2 + O(t^3), u_1 t + u_2 t^2 + O(t^3)) \\
    &= F(\lambda_0, 0) + (F_u(\lambda_0,0)[u_1] + \lambda_1 F_\lambda(\lambda_0,0))t \\
    &\quad + \left(F_u(\lambda_0,0)[u_2] + \frac{1}{2} F_{uu}(\lambda_0,0)[u_1,u_1] + \lambda_2 F_\lambda(\lambda_0,0) \right. \\
    &\qquad \quad \left. + \, \lambda_1 F_{\lambda u}(\lambda_0,0)[u_1] + \frac{1}{2} \lambda_1^2 F_{\lambda \lambda}(\lambda_0,0) \right) t^2 + O(t^3).
  \end{split}
  \]
  Hence we have:
  \[
  F(\lambda_0,0)=0,
  \]
  which is true by hypothesis;
  \[
  F_u(\lambda_0,0)[u_1] + \lambda_1 F_\lambda(\lambda_0,0)=0,
  \]
  which is also true because $F_u(\lambda_0,0)[u_1]=F_u(\lambda_0,0)[u^*]=0$ by hypothesis and $F_\lambda(\lambda_0,0)=0$ because $F(\lambda,0) \equiv 0$ for all $\lambda$;
  \[
  \begin{split}
  0 &= F_u(\lambda_0,0)[u_2] + \frac{1}{2} F_{uu}(\lambda_0,0)[u_1,u_1] + \lambda_2 F_\lambda(\lambda_0,0) \\
    &\qquad \quad + \, \lambda_1 F_{\lambda u}(\lambda_0,0)[u_1] + \frac{1}{2} \lambda_1^2 F_{\lambda \lambda}(\lambda_0,0) \\
    &= F_u(\lambda_0,0)[u_2] + \frac{1}{2} F_{uu}(\lambda_0,0)[u_1,u_1] + \lambda_1 F_{\lambda u}(\lambda_0,0)[u_1],
  \end{split}
  \]
  again because $F(\lambda,0) \equiv 0$ for all $\lambda$. Applying $\psi$ to that last equality, where $\coker F_u(\lambda^*,0) = \coker L = \langle \psi \rangle$, one gets
  \[
  0 = \frac{1}{2} \langle \psi, F_{uu}(\lambda^*,0)[u^*,u^*] \rangle +  \lambda_1 \langle \psi, F_{\lambda u}(\lambda^*,0)[u^*] \rangle,
  \]
  from which we obtain
  \[
  \lambda_1 = -\frac{1}{2} \frac{\langle \psi, F_{uu}(\lambda^*,0)[u^*,u^*] \rangle}{\langle \psi, F_{\lambda u}(\lambda^*,0)[u^*] \rangle}
  \]
  (the fraction is well defined because by the hypothesis of Theorem \ref{thm:1Dbifurcation} we already know that the denominator can't be 0, see also \cite[5.4.3.iv]{AmbrosettiProdi} or \cite[(I.6.3)]{Kielhofer}). If $\lambda_1 \neq 0$ we have then a transcritical bifurcation, proving (1).

  \item Assume now that $\lambda_1=0$. Again, in general, one has the following expansion
  \[
  \begin{split}
  0 &= F(\lambda(t),u(t)) = F(\lambda_0 + \lambda_2 t^2 + \lambda_3 t^3 + O(t^3), u_1 t + u_2 t^2 + u_3 t^3 + O(t^3)) \\
    &= F(\lambda_0, 0) + F_u(\lambda_0,0)[u_1] t \\
    &\quad + \left(F_u(\lambda_0,0)[u_2] + \frac{1}{2} F_{uu}(\lambda_0,0)[u_1,u_1] + \lambda_2 F_\lambda(\lambda_0,0) \right) t^2 \\
    &\quad + \left(F_u(\lambda_0,0)[u_3] + F_{uu}(\lambda_0,0)[u_1,u_2] + \frac{1}{6} F_{uuu}(\lambda_0,0)[u_1,u_1,u_1] \right. \\
    &\qquad \quad + \lambda_3 F_\lambda(\lambda_0,0) + \lambda_2 F_{\lambda,u}(\lambda_0,0)[u_1] \bigg ) t^3 + O(t^4).
  \end{split}
  \]
  As before, the first and the second summands are already known to be 0. The third term gives us the following condition:
  \begin{equation}\label{eq:u2condition}
  F_u(\lambda^*,0)[u_2] + \frac{1}{2} F_{uu}(\lambda^*,0)[u^*,u^*] = 0.
  \end{equation}
  The fourth term, instead, is the one from which we would like to extract the value of $\lambda_2$:
  \[
  \begin{split}
  F_u&(\lambda^*,0)[u_3] + F_{uu}(\lambda^*,0)[u^*,u_2] \\
  &+ \frac{1}{6} F_{uuu}(\lambda^*,0)[u^*,u^*,u^*] + \lambda_2 F_{\lambda,u}(\lambda^*,0)[u^*] = 0,
  \end{split}
  \]
  which leads to
  \begin{equation}\label{eq:lambda2}
  \lambda_2 = -\frac{\langle \psi, F_{uu}(\lambda^*,0)[u^*,u_2] \rangle + \frac{1}{6} \langle \psi, F_{uuu}(\lambda^*,0)[u^*,u^*,u^*] \rangle}{\langle \psi, F_{\lambda,u}(\lambda^*,0)[u^*]\rangle}
  \end{equation}
  (again, observe that the fraction is well defined because the denominator is not 0).

  At least implicitly, then, one can find $\lambda_2$. Indeed, $L: X \rightarrow Y$ factors through
  \[
  \tilde{L}: \frac{X}{\langle u^* \rangle} \rightarrow \{ y \in Y \mid \langle \psi, y \rangle = 0 \},
  \]
  which is invertible, so that we get
  \[
  u_2 = \tilde{L}^{-1} \left(-\frac{1}{2} F_{uu}(\lambda^*,0)[u^*,u^*]\right)
  \]
  to substitute in \eqref{eq:lambda2}. Observe indeed that, applying $\psi$ to \eqref{eq:u2condition}, one immediately gets that
  \[
  F_{uu}(\lambda^*,0)[u^*,u^*] \in \{ y \in Y \mid \langle \psi, y \rangle = 0 \}.
  \]
  Moreover, notice that
  \[
  \langle \psi, F_{uu}(\lambda^*,0)[u^*,v] \rangle = 0, \quad \forall v \in \langle u^* \rangle
  \]
  implies also that the $\lambda_2$ obtained in this way is well defined.

  If $\lambda_2 > 0$ we get a supercritical bifurcation, while if $\lambda_2 < 0$ we have a subcritical bifurcation, proving (2).
\end{enumerate}

\end{proof}

\subsection{Elliptic regularity}
Let $L$ be a linear partial differential operator of order 2 defined in an open subset $\Omega$ of $\mathbb{R}^n$, $n \geq 2$. Assume that $L$ can be written using the standard cartesian coordinates (and Einstein's notation) as follows
\[
Lu = a^{ij}(x)D_{ij} u + b^i(x)D_i u + c(x) u,
\]
with $a^{ij}=a^{ji}$ for each $1 \leq i,j \leq n$.

\begin{theorem} \label{thm:smoothness} \cite[4.2.1]{Shimakura}
If $L$ is an elliptic operator in $\Omega \subset \mathbb{R}^n$ open and if the coefficients of $L$ are of class $C^\infty$ in $\Omega$, then $A$ is \emph{hypoelliptic} in $\Omega$, namely: If $u$ is a distribution in an open subset $\Omega_1$ of $\Omega$ and if $Lu$ is of class $C^\infty$ in $\Omega_1$, then $u$ is of class $C^\infty$ in $\Omega_1$.
\end{theorem}

\begin{definition} \label{def:interiorHolderSpaces}
Let $\Omega \subset \mathbb{R}^n$ be open, $k$ be a non-negative integer and $\alpha \in (0,1)$. Set $d_x := \text{dist}(x, \partial \Omega)$ and $d_{x,y} := \min\{d_x,d_y\}$ and consider $u \in C^k(\Omega)$. We say that $u \in C^{k,\alpha}_*(\Omega)$ if its \emph{interior H\"older norm} is finite, namely
\[
|u|^*_{k,\alpha,\Omega} := |u|^*_{k,\Omega} + [u]^*_{k,\alpha,\Omega} < +\infty,
\]
where
\[
|u|^*_{k,\Omega} := \sum_{j=0}^k [u]^*_{j,\Omega} := \sum_{j=0}^k \sup_{\substack{x \in \Omega \\ |\beta| = j}} d^k_x |D^\beta u(x)|
\]
and
\[
[u]_{k,\alpha,\Omega} := \sup_{\substack{x,y \in \Omega \\ |\beta| = k}} d^{k+\alpha}_{x,y} \frac{|D^\beta u(x) - D^\beta u(y)|}{|x-y|^\alpha}.
\]
\end{definition}

One can prove that $C^{k,\alpha}_*(\Omega)$, equipped with the interior norm, is a Banach space (see for example \cite[Problem 5.2]{GilbargTrudinger}). In order to simplify our notation, from now on we will drop the subscript $*$ and just write $C^{k,\alpha}(\Omega) := C^{k,\alpha}_*(\Omega)$.

In order to state the Schauder's interior estimates we also need to introduce the following norms. Let $\sigma$ be a real number and define
\[
|u|^{(\sigma)}_{k,\alpha,\Omega} := |u|^{(\sigma)}_{k,\Omega} + [u]^{(\sigma)}_{k,\alpha,\Omega} < +\infty,
\]
where
\[
|u|^{(\sigma)}_{k,\Omega} := \sum_{j=0}^k [u]^{(\sigma)}_{j,\Omega} := \sum_{j=0}^k \sup_{\substack{x \in \Omega \\ |\beta| = j}} d^{j+\sigma}_x |D^\beta u(x)|
\]
and
\[
[u]^{(\sigma)}_{k,\alpha,\Omega} := \sup_{\substack{x,y \in \Omega \\ |\beta| = k}} d^{k+\alpha+\sigma}_{x,y} \frac{|D^\beta u(x) - D^\beta u(y)|}{|x-y|^\alpha}.
\]

One can check the following \cite[6.11]{GilbargTrudinger}.
\begin{proposition} \label{prop:CauchySchwarz}
Let $\sigma + \tau \geq 0$, then
\[
|fg|^{(\sigma+\tau)}_{0,\alpha,\Omega} \leq |f|^{(\sigma)}_{0,\alpha,\Omega} |g|^{(\tau)}_{0,\alpha,\Omega}.
\]
\end{proposition}

The basic Schauder's interior estimates are provided by the following
\begin{theorem} \label{thm:SchauderInterior} \cite[6.2]{GilbargTrudinger}
Let $\Omega$ be an open subset of $\mathbb{R}^n$ and let $u \in C^{2,\alpha}(\Omega)$ be a bounded solution in $\Omega$ of
\[
Lu = a^{ij}(x) \partial_i \partial_j u + b^i(x) \partial_i u + c(x) u = f,
\]
where $f \in C^{0,\alpha}(\Omega)$ and there are positive constants $\lambda, \Lambda$ such that the coefficients satisfy
\[
a^{ij}(x) \xi_i \xi_j \geq \lambda |\xi|^2 \qquad \forall x \in \Omega, \xi \in \mathbb{R}^n
\]
and
\[
|a^{ij}|^{(0)}_{0,\alpha,\Omega}, |b^i|^{(1)}_{0,\alpha,\Omega}, |c|^{(2)}_{0,\alpha,\Omega} \leq \Lambda.
\]
Then there exists a constant $C>0$ not depending on $u$ and $f$ such that
\[
|u|^*_{2,\alpha,\Omega} \leq C\left(|u|_{0,\Omega} + |f|^{(2)}_{0, \alpha, \Omega}\right).
\]
\end{theorem}

The following is a generalization of Theorem \ref{thm:SchauderInterior} to the case in which we consider ``weighted'' H\"older norms. Let $w \in C^{k,\alpha}(\Omega), w > 0$ be the weight. Define, according to the notation explained before, the space
\[
C^{k,\alpha}_w(\Omega) := \left\{ u \in C^k(\Omega) \middle| |wu|^*_{k,\alpha,\Omega} < +\infty \right\}.
\]

\begin{theorem}\label{thm:weightedSchauder} \textsc{(Weighted Schauder's Estimates)}
Let $\Omega$ be an open subset of $\mathbb{R}^n$ and let $w \in C^{2,\alpha}(\Omega), w>0$ be a weight such that there exists a positive constant $K$ such that
\[
|\partial_j \log w|^{(1)}_{0,\alpha,\Omega}, \left|\frac{\partial_i \partial_j w}{w}\right|^{(2)}_{0,\alpha,\Omega} \leq K, \quad \forall i,j.
\]
Let $u \in C^{2,\alpha}_w(\Omega)$ be a solution in $\Omega$ of
\[
Lu = a^{ij}(x) \partial_i \partial_j u + b^i(x) \partial_i u + c(x) u = f,
\]
where $f \in C^{0,\alpha}_w(\Omega)$ and there are positive constants $\lambda, \Lambda$ such that the coefficients satisfy
\[
a^{ij}(x) \xi_i \xi_j \geq \lambda |\xi|^2 \qquad \forall x \in \Omega, \xi \in \mathbb{R}^n
\]
and
\[
|a^{ij}|^{(0)}_{0,\alpha,\Omega}, |b^i|^{(1)}_{0,\alpha,\Omega}, |c|^{(2)}_{0,\alpha,\Omega} \leq \Lambda.
\]
Assume moreover that $w u$ is bounded. Then there exists a constant $C>0$ not depending on $u$ and $f$ such that
\[
|u|^*_{2,\alpha,\Omega;w} \leq C\left(|u|_{0,\Omega;w} + |f|^{(2)}_{0, \alpha, \Omega;w}\right),
\]
where $|u|^*_{2,\alpha,\Omega;w} := |wu|^*_{2,\alpha,\Omega}$, $|u|_{0,\Omega;w} := |wu|_{0,\Omega}$ and $|f|^{(2)}_{0, \alpha, \Omega;w} := |wf|^{(2)}_{0, \alpha, \Omega}$.
\end{theorem}

For the proof, see Appendix \ref{app:weightedSchauder}.

\section{Dimension 2}\label{sec:2D}
We want to find some non-trivial solutions of the 2-dimensional Liouville equation
\[
\Delta u(x,y) + \text{e}^{u(x,y)} = 0, \qquad \forall (x,y) \in \mathbb{R}^2,
\]
To use a bifurcation theorem, we first need to find a trivial solution, i.e.~a solution depending only on one variable, say $x$, and constant in the other. The equation then becomes the ordinary differential equation
\[
u''(x) + \text{e}^{u(x)} = 0, \qquad \forall x \in \mathbb{R}^2,
\]
which admits the family of solutions
\[
\log\left[c_1 - c_1\tanh^2\left(\frac{1}{2}\sqrt{2c_1\left(c_2+x\right)^2}\right)\right], \quad c_1 \geq 0, \quad c_2 \in \mathbb{R}.
\]
Observe that the two parameters account only for a translation and a dilation of the solution, so we can simply fix them. For our convenience, we choose $c_1 = 2$ and $c_2 = 0$ (namely, we are requiring that the solution is even and we are fixing its volume), getting
\[
u_0(x,y) := \log[2(1 - \tanh^2(|x|))] = \log(2\sech^2(x)).
\]
We now need to specify the setting we are working in. As we know that the trivial solution does not depend on the $y$ variable, we could just choose to restrict the problem to the subsets
\[
S_\lambda := \{(x,y) \in \mathbb{R}^2 \mid x \in \mathbb{R}, y \in (0,\lambda)\}, \quad \lambda > 0.
\]
In this way, we get a natural parameter for our problem: the witdth $\lambda$ of the strip $S_\lambda$. We now need to fix the Banach spaces: we ask that the solution goes to $-\infty$ as $|x|\rightarrow +\infty$. Consequently, the growth must be at most linear (the equation essentially says that at $\infty$ the second derivative must be 0). In particular, one might ask that the perturbation is bounded by a function that grows slower than $|x|$, like $\sqrt{|x|}$. Define then
\[
X_\lambda := \left\{ u \in C^{2,\alpha}(S_\lambda) \, \middle| \begin{aligned} &\, \frac{\partial}{\partial y} u(x,0) = \frac{\partial}{\partial y} u(x,\lambda) = 0 \quad \forall x \in \mathbb{R}, \\
    &u(-x,y)=u(x,y) \quad \forall (x,y) \in S_\lambda, \\
    &\left|\langle x \rangle^{-\frac{1}{2}} u\right|_{2,\alpha,S_\lambda} + \left|\langle x \rangle^{\frac{3}{2}} \Delta u\right|_{0,\alpha,S_\lambda} < +\infty \end{aligned} \right\}
\]
and
\[
Y_\lambda := \left\{u \in C^{0,\alpha}(S_\lambda) \middle| \begin{aligned} \,
    &u(-x,y)=u(x,y) \quad \forall (x,y) \in S_\lambda, \\
    &\left|\langle x \rangle^{\frac{3}{2}} f\right|_{0,\alpha,S_\lambda} < +\infty \end{aligned}
\right\},
\]
where $\langle x \rangle := \sqrt{1+x^2}$. Recalling that the interior H\"older spaces of Definition \ref{def:interiorHolderSpaces} are Banach spaces (\cite[Problem 5.2]{GilbargTrudinger}), it can be easily checked that both $X_\lambda$ and $Y_\lambda$ are Banach spaces when endowed, respectively, with the norms
\[
\norm{u}_{X_\lambda} := \left|\langle x \rangle^{-\frac{1}{2}} u\right|_{2,\alpha,S_\lambda} + \left|\langle x \rangle^{\frac{3}{2}} \Delta u\right|_{0,\alpha,S_\lambda}
\]
and
\[
\norm{f}_{Y_\lambda} := \left|\langle x \rangle^{\frac{3}{2}} f\right|_{0,\alpha,S_\lambda}.
\]
Observe moreover that the functions in $X_\lambda$ grow at most as $\sqrt{|x|}$, while those in $Y_\lambda$ grow at most as $|x|^{-\frac{3}{2}}$.

Our problem is then finding the zeros of the following function:
\[
\begin{split}
\tilde{F} : X_\lambda &\longrightarrow Y_\lambda \\
    u &\longmapsto \Delta (u_0 + u) + \text{e}^{u_0+u} = \Delta u + \text{e}^{u_0} (\text{e}^u - 1).
\end{split}
\]

\begin{remark}
  Once we have a non-trivial solution on the strip, we can obtain a non-trivial solution on the whole plane by extending the solution reflecting it repeatedly along the borders of the strip. Continuity is automatic by construction and Neumann conditions assure that this is still a weak solution of the problem. Bootstrapping and weighted elliptic regularity (Theorem \ref{thm:weightedSchauder}) make this also a strong solution.
\end{remark}

\begin{lemma}\label{lemma:candidateLambda}
  A necessary condition for $\lambda$ to be a bifurcation point is that $\lambda = \pi k$ for some $k \in \mathbb{N}_{>0}$. Moreover, in that case we have
  \[
  \ker L = \{t \sech(x) \cos(y) \mid t \in \mathbb{R}\}.
  \]
\end{lemma}

\begin{proof}
Recall that by Proposition \ref{prop:bifNecessaryCondition} a necessary condition to have a bifurcation on $S_\lambda$ is that the linearized operator in $u_0$ on the strip $S_\lambda$ has a non-trivial kernel.

The linearization of $\tilde{F}(u) = \Delta u + \text{e}^{u_0} (\text{e}^u - 1)$ in the point 0 is
\[
L(v)(x,y) = \Delta v(x,y) + 2 \sech^2(x) v(x,y).
\]
We might first want to look for solutions of $Lv=0$ having the form
\[
v(x,y) = w_1(x) w_2(y),
\]
hence satisfying
\[
w_1''(x) w_2(y) + w_1(x) w_2''(y) + 2 \sech^2(x) w_1(x) w_2(y) = 0.
\]
By separation of variables we get $w_2(y) = A \cos(\mu y) + B \sin(\mu y)$. Imposing Neumann boundary conditions we get $B=0$ and $\mu = \frac{\pi j}{\lambda}, \, j \in \mathbb{Z}$. Hence, we can directly look for solutions of the form
\[
v_j(x,y) = \cos\left(\frac{\pi j}{\lambda}y\right) \tilde{v}_j(x),
\]
where $j \in \mathbb{N}$ is fixed, which means that we need to solve the problem:
\[
-\tilde{v}_j''(x) - 2\sech^2(x) \tilde{v}_j(x) = -\left(\frac{\pi j}{\lambda}\right)^2 \tilde{v}_j(x), \quad \forall	x \in \mathbb{R}, \, \forall j \in \mathbb{N}.
\]
Observe that the last equation is a stationary Schr\"odinger equation, with a so-called \textit{P\"oschl-Teller potential} \cite{PoschlTeller}. Some lenghty but not difficult considerations on the growth of Legendre functions lead to
\[
\tilde{v}_j \not\equiv 0 \iff -\left(\frac{\pi j}{\lambda}\right)^2 = -1 \iff j = \frac{\lambda}{\pi}
\]
(see Appendix \ref{app:2Dcomputations}). We then get that the only nonzero element of the family is
\[
v_{\frac{\lambda}{\pi}}(x,y) = A \sech(x) \cos(y).
\]
Summing up, we found that the only values of $\lambda$ for which we can expect to have a bifurcation are the points $\pi j$, with $j \in \mathbb{N}_{>0}$. For these values of $\lambda$, in particular, the operator $L$ has one dimensional kernel:
\[
\ker L = \{t \sech(x) \cos(y) \mid t \in \mathbb{R}\}.
\]
\end{proof}

Lemma \ref{lemma:candidateLambda} restricts the set of candidate bifurcation points to just $\lambda = \pi k$ for $k \in \mathbb{N}_{>0}$. In the following we will prove that all of them are actually bifurcation points (and that the bifurcations are all supercritical).

\begin{lemma}\label{lemma:linearizedFredholm}
  The linearized operator $L$ is Freholm of index 0.
\end{lemma}

\begin{proof}
Observe first that $Y_\lambda \subset L^2(S_\lambda)$, so that
\[
\codim L = \dim \{ f \in Y_\lambda \mid \langle Lu, f \rangle_{L^2} = 0 \quad \forall u \in X_\lambda\}.
\]
Indeed, if $[f] \in \coker L$, then $[f]$ has a representative which is perpendicular to $R(L)$ (just take its projection on $R(L)^\perp$), while if $f$ is in the set on the right hand side then $f$ cannot be parallel to any element of $R(L)$ (otherwise the scalar product with such an element would not be zero) and thus cannot be in $R(L)$, which is a vector subspace of $Y_\lambda$ because $L$ is linear.

Let then $f \in Y_\lambda$ be such that $\langle L u , f \rangle_{L^2} = 0$ for each $u \in X_\lambda$. By Theorem \ref{thm:smoothness}, we have that $f$ is $C^\infty$ on $S_\lambda$.

We now only need to prove that $L$ is symmetric on the elements $f$ in $Y_\lambda$ such that $\langle L u , f \rangle_{L^2} = 0$ for each $u \in X_\lambda$. Take $u \in X_\lambda$ and $f \in Y_\lambda \cap C^2(S_\lambda) \subset L^2(S_\lambda)$ such that $Lf=0$. Then
\begin{equation} \label{eq:explicitSelfAdjointess}
\begin{split}
\langle Lu, &f \rangle_{L^2} = \int_{S_\lambda} \left( \Delta u + \text{e}^{u_0} u \right) f \\
    &= \int_{(0,\lambda) \times \mathbb{R}} \left( \pd[2]{u}{x}(x,y) + \pd[2]{u}{y}(x,y) + 2 \sech^2(x) u(x,y) \right) f(x,y) \dif x \dif y.
\end{split}
\end{equation}
The last summand inside the integral is clearly symmetric. Thus, we can just look at the first two. Theorem \ref{thm:weightedSchauder} applied to $Lf=0$ gives
\begin{equation} \label{eq:Schauder}
\left|\langle x \rangle^{\frac{3}{2}} f\right|_{2,\alpha,S_\lambda} \leq C' \left|\langle x \rangle^{\frac{3}{2}} f\right|_{0,\alpha,S_\lambda} \leq C
\end{equation}
for some constant $C > 0$ (recall that $f \in Y_\lambda$).  Hence, both the first order and the second order partial derivatives in $x$ decay as $|x|^{-\frac{3}{2}}$ as $|x| \rightarrow \infty$. Indeed
\[
\pd{}{x} \left( \langle x \rangle^\frac{3}{2} f(x,y) \right) = \frac{3}{2} x \langle x \rangle^{-\frac{1}{2}} f(x,y) + \langle x \rangle^\frac{3}{2} \pd{f}{x}(x,y)
\]
and consequently, making use of \eqref{eq:Schauder},
\begin{equation} \label{eq:firstDerivativeDecay}
\begin{split}
\sup_{(x,y) \in S_\lambda} \left| \langle x \rangle^\frac{3}{2} \pd{f}{x}(x,y) \right| &\leq C + \sup_{(x,y) \in S_\lambda} \left| \frac{3}{2} x \langle x \rangle^{-\frac{1}{2}} f(x,y) \right| \\
    &\leq C + \frac{3}{2} \sup_{(x,y) \in S_\lambda} \left| \langle x \rangle^\frac{3}{2} f(x,y) \right| \leq C_1
\end{split}
\end{equation}
for some $C_1 > 0$ (which can be computed explicitly). Analogously, the second derivative turns out to be
\[
\begin{split}
\pd[2]{}{x} &\left( \langle x \rangle^\frac{3}{2} f(x,y) \right) \\
    &= \frac{3}{4}(x^2 + 2) \langle x \rangle^{-\frac{5}{2}}f(x,y) + 3x \langle x \rangle^{-\frac{1}{2}} \pd{f}{x}(x,y) + \langle x \rangle^\frac{3}{2} \pd[2]{f}{x}(x,y).
\end{split}
\]
With estimates similar to the above and using \eqref{eq:firstDerivativeDecay} we then find
\begin{equation} \label{eq:secondDerivativeDecay}
\sup_{(x,y) \in S_\lambda} \left| \langle x \rangle^\frac{3}{2} \pd[2]{f}{x}(x,y) \right| \leq C_2
\end{equation}
for some $C_2 > 0$ (which, again, can be computed explicitly).

A first consequence of \eqref{eq:firstDerivativeDecay} and \eqref{eq:secondDerivativeDecay}, together with the H\"older inequality, is that in what follows we can always apply Fubini's Theorem. Let us look then at the first summand in \eqref{eq:explicitSelfAdjointess}.
\[
\begin{split}
&\int_{S_\lambda} \frac{\partial^2 u}{\partial x^2}(x,y) f(x,y) \dif x \dif y \\
&= \int_0^\lambda \int_{-\infty}^{+\infty} \left[ \frac{\partial}{\partial x} \left( \frac{\partial u}{\partial x}(x,y) f(x,y) \right) - \frac{\partial u}{\partial x}(x,y) \frac{\partial f}{\partial x}(x,y) \right] \dif x \dif y \\
&= \int_0^\lambda \left[\frac{\partial u}{\partial x}(x,y) f(x,y)\right]_{x=-\infty}^{+\infty} \dif x - \int_{S_\lambda} \frac{\partial u}{\partial x}(x,y) \frac{\partial f}{\partial x}(x,y) \dif x \dif y \\
&= - \int_0^\lambda \left[ u(x,y) \frac{\partial f}{\partial x}(x,y) \right]_{x=-\infty}^{+\infty} \dif x + \int_{S_\lambda} u(x,y) \frac{\partial^2 f}{\partial x^2}(x,y) \dif x \dif y \\
&= \int_{S_\lambda} u(x,y) \frac{\partial^2 f}{\partial x^2}(x,y) \dif x \dif y
\end{split},
\]
where, in order to pass from the 3-rd to the 4-th to the 5-th line, we used the fact that both $\frac{\partial u}{\partial x} f$ and $u \frac{\partial f}{\partial x}$ decay as $|x|^{-2}$ as $|x| \rightarrow \infty$ (the estimate for $\frac{\partial u}{\partial x} f$ is obtained in the same way as in \eqref{eq:firstDerivativeDecay}).

As for the second summand, the computations are similar but easier, as we can directly exploit the fact that we are imposing Neumann conditions on the boundary of $S_\lambda$ (which, recall, is $\mathbb{R} \times (0,\lambda)$). Hence, again the boundary terms go away while integrating by parts and so, putting all together, we obtain that $L$ is symmetric.

By the same argument of the Lemma \ref{lemma:candidateLambda}, we get then that, if $\lambda \neq \pi k$ for some $k \in \mathbb{N}_{>0}$, then $\coker L$ is trivial. Instead, if $\lambda = \pi k$ for some $k \in \mathbb{N}_{>0}$, then there exists only one family of solutions of $\langle L u , f \rangle = 0, \forall u \in X_\lambda$ satisfying Neumann conditions on $\partial S_\lambda$ and belonging to $Y_\lambda$, namely $\langle u^* \rangle \subset X_\lambda \cap Y_\lambda$. Consequently, $\codim L = 1$. In both cases, $L$ is Freholm of index 0, as wanted.
\end{proof}

\begin{proof}[Proof of Theorem \ref{thm:2Dnon-trivialSolution}]
We just have to apply Theorem \ref{thm:1Dbifurcation}. The change of variables that changes $y$ into $\lambda y$ and leaves $x$ unchanged transforms our original problem into one to which we can apply Theorem \ref{thm:1Dbifurcation}:
\[
 F(\lambda, u) = \pd[2]{u}{x} + \frac{1}{\lambda^2} \pd[2]{u}{y} + \text{e}^{u_0}\left(\text{e}^u - 1\right),
\]
where $F: (0,+\infty) \times X_1 \rightarrow Y_1$.

Take then $\lambda^* = \pi j$, with $j \in \mathbb{N}_{>0}$. First,
\[
 F_u(\lambda^*, 0)[v] = \pd[2]{v}{x} + \frac{1}{(\lambda^*)^2}\pd[2]{v}{y} + \text{e}^{u_0} v
\]
is Fredholm, because so it was before the change of variables. Secondly,
\[
 M[v] := F_{u, \lambda}(\lambda^*, 0)[1,v] = -\frac{1}{(\lambda^*)^3} \pd[2]{v}{y}
\]
and, as we know that
\[
 v_0(x,y) = \sech(x) \cos(y) \in \ker L,
\]
we have that
\[
 \ker F_u(\lambda^*, 0) = \langle u^* \rangle
\]
with
\[
 u^*(x,y) := R_{\lambda^*}(v_0)(x,y) = \sech(x) \cos(\lambda^* y) \in X_1.
\]
Hence
\[
 M[u^*] = \frac{2}{\lambda^*} u^*(x,y) \in \langle u^* \rangle,
\]
so that $M[u^*] \not\in R$.

According to Theorem \ref{thm:1Dbifurcation}, then, $\lambda^* = \pi j$ for $j \in \mathbb{N}_{>0}$ are bifurcation points for $F$ and, consequently, also for $\tilde{F}$ (which is nothing less than $F$ written using different coordinates).

Finally, Proposition \ref{prop:bifShape} gives that all bifurcations are supercritical. This means that the period of the non-trivial solutions we find in this way is actually slightly bigger than $\pi k$ (for $\lambda = \pi k$ the trivial branch and the non-trivial branch meet, so that they identify the same trivial solution). Notice though that we can actually make the period be exactly $\pi k$ by replacing the non-trivial solution $u$ of period $(1 + \varepsilon) \pi k$ with
\[
u \left(\frac{x}{1 + \varepsilon}, \frac{y}{1 + \varepsilon} \right) + 2\log\frac{1}{1+\varepsilon},
\]
which is still a solution of $\Delta u + \text{e}^u = 0$, of period $\pi k$.
\end{proof}

\section{Dimension 4}\label{sec:4D}
In this section, we prove the same result of the previous section in the 4D case. The main tool we use is Krasnosel'skii's Theorem \ref{thm:krasnoselskii}.

\subsection{Trivial solutions} \label{subsec:trivial}
Our first goal is to show that there exists at least one solution of
\begin{equation} \label{eq:3Dtrivial}
\begin{cases}
\Delta^2 u = \text{e}^u \text{ in } \mathbb{R}^3 \\
\int_{\mathbb{R}^3} \text{e}^{u(x)} \dif x < +\infty
\end{cases}.
\end{equation}
The proof will be done in two steps and we will look in particular for solutions of the integral form of \eqref{eq:3Dtrivial}, namely solutions of
\[
u(x) = -\frac{1}{8\pi} \int_{\mathbb{R}^3} |x-y| \text{e}^{u(y)} \dif y.
\]
Observe that a function $u$ with finite volume satisfying this last expression is a solution of \eqref{eq:3Dtrivial}. Indeed, a fundamental solution of $\Delta^2$ is $G(x) = -\frac{1}{8\pi} |x|$ (see \cite{Cohl2013}).

\begin{lemma}
Let
\[
X := \left\{ u \in C^0(\mathbb{R}^3) \,\middle|\, u \text{ is radially symmetric and } \norm{u} < +\infty \right\},
\]
where $\norm{u} := \sup_{x \in \mathbb{R}^3} \frac{|u(x)|}{1+|x|}$. Then for every $\varepsilon > 0$ there exist $u_\varepsilon \in X$ such that
\begin{equation} \label{eq:regularizedIntegral3D}
u_\varepsilon(x) = -\frac{1}{8\pi} \int_{\mathbb{R}^3} |x-y| \, \text{\emph{e}}^{-\varepsilon |y|^2} \text{\emph{e}}^{u_\varepsilon(y)} \dif y.
\end{equation}
\end{lemma}

\begin{proof}
First of all, observe that $X$, endowed with the norm $\norm{\cdot}$, is a well-defined Banach space. Define then
\[
\begin{split}
T_\varepsilon:\, X &\longrightarrow X \\
    u &\longmapsto \bar{u}_\varepsilon, \quad \bar{u}_\varepsilon(x) := -\frac{1}{8\pi} \int_{\mathbb{R}^3} |x-y| \text{e}^{-\varepsilon |y|^2} \text{e}^{u(y)} \dif y
\end{split} \quad .
\]
$T_\varepsilon$ is well defined. Take in fact $u \in X$, then $\bar{u}_\varepsilon \in C^0(\mathbb{R}^3)$ by the Lebesgue Dominated Convergence Theorem. Moreover, $\bar{u}_\varepsilon$ is clearly radial because of the radial invariance of the Lebesgue integral: indeed, if $A \in SO(3)$, then
\[
\begin{split}
\bar{u}_\varepsilon(Ax) &= -\frac{1}{8\pi} \int_{\mathbb{R}^3} |Ax - y| \, \text{e}^{-\varepsilon |y|^2} \text{e}^{u(y)} \dif y \\
    &= -\frac{1}{8\pi} \int_{\mathbb{R}^3} |A(x - y)| \, \text{e}^{-\varepsilon |Ay|^2} \text{e}^{u(Ay)} |\det A| \dif y \\
    &= -\frac{1}{8\pi} \int_{\mathbb{R}^3} |x - y| \, \text{e}^{-\varepsilon |y|^2} \text{e}^{u(y)} \dif y = \bar{u}_\varepsilon(x).
\end{split}
\]
Finally,
\[
\begin{split}
|\bar{u}_\varepsilon(x)| &= \frac{1}{8\pi} \int_{\mathbb{R}^3} |x-y| \, \text{e}^{-\varepsilon |y|^2} \text{e}^{u(y)} \dif y \\
    &\leq \frac{1}{8\pi} \int_{\mathbb{R}^3} |x-y| \, \text{e}^{-\varepsilon |y|^2} \text{e}^{\,\norm{u}(1+|y|)} \dif y \\
    &\leq \frac{1}{8\pi} \int_{\mathbb{R}^3} |y| \text{e}^{-\varepsilon |y|^2} \text{e}^{\,\norm{u}(1+|y|)} \dif y + |x| \frac{1}{8\pi} \int_{\mathbb{R}^3} \text{e}^{-\varepsilon |y|^2} \text{e}^{\,\norm{u}(1+|y|)} \dif y \\
    &\leq C_1 + C_2 |x| \leq \bar{C}(1+|x|),
\end{split}
\]
so that $\norm{\bar{u}_\varepsilon} < +\infty$. Hence $T_\varepsilon(u) = \bar{u}_\varepsilon \in X$.

We now show that $T_\varepsilon$ is compact. Take a bounded sequence $\{u_n\}_n \subset X$, $\norm{u_n} \leq C < +\infty$ for all $n \in \mathbb{N}$. We want to show that then $\{T_\varepsilon(u_n)\}_n$ admits a converging subsequence. The idea is to use Arzelà-Ascoli's Theorem on the sequence $\left\{ \frac{T_\varepsilon(u_n)}{1+|x|} \right\}_n$. First,
\[
\begin{split}
\frac{|T_\varepsilon(u_n)|}{1+|x|} &= \frac{1}{8\pi (1+|x|)} \int_{\mathbb{R}^3} |x-y| \, \text{e}^{-\varepsilon |y|^2} \text{e}^{u_n(y)} \dif y \\
    &\leq \frac{1}{8\pi (1+|x|)} \int_{\mathbb{R}^3} |x-y| \, \text{e}^{-\varepsilon |y|^2} \text{e}^{C(1+|y|)} \dif y \\
    &\leq \frac{C_1 + C_2 |x|}{8\pi (1+|x|)} \leq \bar{C} < +\infty,
\end{split}
\]
for any $x \in \mathbb{R}^3$ and $n \in \mathbb{N}$, so that $\left\{ \frac{T_\varepsilon(u_n)}{1+|x|} \right\}_n$ is uniformly bounded. Moreover,
\[
\begin{split}
\left| \frac{T_\varepsilon(u_n(x))}{1+|x|} - \frac{T_\varepsilon(u_n(y))}{1+|y|} \right| &= \frac{1}{8\pi} \left| \int_{\mathbb{R}^3} \left( \frac{|x-z|}{1+|x|} - \frac{|y-z|}{1+|y|} \right) \text{e}^{-\varepsilon |z|^2} \text{e}^{u_n(z)} \dif z \right| \\
    &\leq \frac{1}{8\pi} \int_{\mathbb{R}^3} \left| \frac{|x-z|}{1+|x|} - \frac{|y-z|}{1+|y|} \right| \text{e}^{-\varepsilon |z|^2} \text{e}^{u_n(z)} \dif z \\
    &\leq \left( \frac{1}{8\pi} \int_{\mathbb{R}^3} (2+|z|) \text{e}^{-\varepsilon |z|^2} \text{e}^{C(1+|z|)} \dif z \right) |x-y|
\end{split}
\]
for any $x,y \in \mathbb{R}^3$ and $n \in \mathbb{N}$, so that $\left\{\frac{T_\varepsilon(u_n(x))}{1+|x|}\right\}_n$ is equicontinuous. The last inequality, in particular follows from the triangular inequality: as one clearly has
\[
\begin{split}
|y| &\leq |x| + |y-x|, \\
|x-z| &\leq |x-y| + |y-z|, \\
|z| &\leq |x| + |z-x|, \\
|x-z| &\leq |x| + |z|,
\end{split}
\]
we obtain indeed
\[
\begin{split}
\left| \frac{|x-z|}{1+|x|} - \frac{|y-z|}{1+|y|}\right| &= \left| \frac{|x-z| - |y-z| + |y||x-z| - |x||y-z|}{(1+|x|)(1+|y|)} \right| \\
    &\leq \frac{|x-y| + (|x|+|y-x|)|x-y| - |x||y-z|}{(1+|x|)(1+|y|)} \\
    &\leq \frac{|x-y| + |x||x-y| + |x-z||x-y|}{(1+|x|)(1+|y|)} \\
    &\leq \left(\frac{1+2|x|}{(1+|x|)(1+|y|)} + \frac{|z|}{(1+|x|)(1+|y|)} \right) |x-y| \\
    &\leq \left(\frac{2}{(1+|y|)} + \frac{|z|}{(1+|x|)(1+|y|)} \right) |x-y| \\
    &\leq (2 + |z|) |x-y|.
\end{split}
\]
By Arzelà-Ascoli's Theorem, $\left\{\frac{T_\varepsilon(u_n(x))}{1+|x|}\right\}_n$ admits a subsequence which converges uniformly. Thus, $\{ T_\varepsilon(u_n) \}_n$ admits a converging subsequence in $(X, \norm{\cdot})$ and therefore $T$ is a compact operator.

Next, we prove that $T$ has a fixed point using Schaefer's Fixed Point Theorem (see for example \cite{Zeidler1986}). Let $u \in X$ satisfy $u = tT_\varepsilon(u)$ for some $0 \leq t \leq 1$, then
\[
u(x) = -\frac{t}{8\pi} \int_{\mathbb{R}^3} |x-y| \text{e}^{-\varepsilon |y|^2} \text{e}^{u(y)} \dif y \leq 0.
\]
Consequently
\[
|u(x)| \leq \frac{t}{8\pi} \int_{\mathbb{R}^3} |x-y| \text{e}^{-\varepsilon |y|^2} \dif y \leq C(1+|x|)
\]
and therefore $\norm{u} \leq C$. That means that the set $\{u \in X \mid u = tT_\varepsilon(u), 0 \leq t \leq 1 \}$ is bounded in $(X, \norm{\cdot})$: by Schaefer's Theorem then $T_\varepsilon$ has a fixed point in $X$.
\end{proof}

\begin{theorem}\label{thm:trivialSolution}
$u_\varepsilon$ converges to some $u$ in $(X, \norm{\cdot})$ as $\varepsilon$ goes to 0, with $u$ satisfying
\[
u(x) = -\frac{1}{8\pi} \int_{\mathbb{R}^3} |x-y| \, \text{\emph{e}}^{u(y)} \dif y
\]
(hence being a solution of \eqref{eq:3Dtrivial} with the desired properties).
\end{theorem}

Before proving the Theorem, we need to establish the following Pohozaev-type identity \cite{Pohozaev1965}.

\begin{lemma}\label{lemma:Pohozaev}
  Let $u_\varepsilon$ be a solution of \eqref{eq:regularizedIntegral3D}, namely
  \[
  u_\varepsilon(x) = -\frac{1}{8\pi} \int_{\mathbb{R}^3} |x-y| \, \text{\emph{e}}^{-\varepsilon |y|^2} \text{\emph{e}}^{u_\varepsilon(y)} \dif y.
  \]
  Then
  \[
  \int_{\mathbb{R}^3} \left( u_\varepsilon(x) + 6 - 4\varepsilon|x|^2 \right) \text{e}^{-\varepsilon |x|^2} \text{e}^{u_\varepsilon(x)} \dif x = 0.
  \]
\end{lemma}
\begin{proof}[Proof of Lemma \ref{lemma:Pohozaev}]
This proof follows the ideas of \cite{Hyder2019}.

Differentiating \eqref{eq:regularizedIntegral3D} one gets
\[
x \cdot \nabla u_\varepsilon(x) = -\frac{1}{8\pi} \int_{\mathbb{R}^3} \frac{x \cdot (x-y)}{|x-y|} \text{e}^{-\varepsilon|y|^2 + u_\varepsilon(y)} \dif y.
\]
Multiplying both sides by $\text{e}^{-\varepsilon|x|^2 + u_\varepsilon(x)}$ and integrating on $x$ we obtain
\begin{equation}\label{eq:pohozaevIntegratedDifferential}
  \begin{split}
    \int_{\mathbb{R}^3} &\left(x \cdot \nabla u_\varepsilon(x)\right) \text{e}^{-\varepsilon|x|^2 + u_\varepsilon(x)} \dif x \\
    &= -\frac{1}{8\pi} \int_{\mathbb{R}^3} \int_{\mathbb{R}^3} \frac{x \cdot (x-y)}{|x-y|} \text{e}^{-\varepsilon|y|^2+u_\varepsilon(y)} \text{e}^{-\varepsilon|x|^2+u_\varepsilon(x)} \dif y \dif x.
  \end{split}
\end{equation}
Integrating by parts the LHS one gets
\begin{equation}\label{eq:pohozaevLHS}
  \begin{split}
    \text{(LHS)} &= \int_{\mathbb{R}^3} \sum_{j=1}^3 x_j \pd{u_\varepsilon(x)}{x_j} \text{e}^{-\varepsilon|x|^2 + u_\varepsilon(x)} \dif x \\
    &= \sum_{j=1}^3 \int_{\mathbb{R}^3} x_j \text{e}^{-\varepsilon|x|^2} \pd{}{x_j}\left(\text{e}^{u_\varepsilon(x)}\right) \dif x \\
    &= - \sum_{j=1}^3 \int_{\mathbb{R}^3} \pd{}{x_j} \left(x_j \text{e}^{-\varepsilon|x|^2}\right) \text{e}^{u_\varepsilon(x)} \dif x \\
    &= -\sum_{j=1}^3 \int_{\mathbb{R}^3} \left( \text{e}^{-\varepsilon|x|^2} - 2\varepsilon x_j^2 \text{e}^{-\varepsilon|x|^2} \right) \text{e}^{u_\varepsilon(x)} \dif x \\
    &= -3 \int_{\mathbb{R}^3} \text{e}^{-\varepsilon|x|^2 + u_\varepsilon(x)} \dif x + \int_{\mathbb{R}^3} 2|x|^2 \varepsilon \text{e}^{-\varepsilon|x|^2 + u_\varepsilon(x)} \dif x.
  \end{split}
\end{equation}

For the RHS, instead, notice that $x = \frac{(x+y)+(x-y)}{2}$. Then
\begin{equation}\label{eq:pohozaevRHS}
  \begin{split}
    \text{(RHS)} &= \frac{1}{2} \int_{\mathbb{R}^3} \text{e}^{-\varepsilon|x|^2 + u_\varepsilon(x)} \left( -\frac{1}{8\pi} \int_{\mathbb{R}^3} |x-y| \text{e}^{-\varepsilon|y|^2 + u_\varepsilon(y)} \dif y \right) \dif x \\
    &\qquad - \frac{1}{16\pi} \int_{\mathbb{R}^3} \int_{\mathbb{R}^3} \frac{(x+y) \cdot (x-y)}{|x-y|} \text{e}^{-\varepsilon|x|^2 + u_\varepsilon(x) - \varepsilon|y|^2 + u_\varepsilon(y)} \dif y \dif x \\
    &= \frac{1}{2} \int_{\mathbb{R}^3} u_\varepsilon(x) \text{e}^{-\varepsilon|x|^+u_\varepsilon(x)} \dif x,
  \end{split}
\end{equation}
where the second term in the sum vanishes because the integrand is skew-symmetric in $(x,y)$.

Plugging \eqref{eq:pohozaevLHS} and \eqref{eq:pohozaevRHS} in \eqref{eq:pohozaevIntegratedDifferential} one finally gets
\[
\int_{\mathbb{R}^3} \left( u_\varepsilon(x) + 6 - 4\varepsilon|x|^2 \right) \text{e}^{-\varepsilon |x|^2} \text{e}^{u_\varepsilon(x)} \dif x = 0.
\]
\end{proof}

\begin{proof}[Proof of Theorem \ref{thm:trivialSolution}]
First of all, we check that $u_\varepsilon$ is monotone decreasing for each $\varepsilon>0$. Indeed, write the integral in $u_\varepsilon$ in polar coordinates (with a slight abuse of notation)
\[
\begin{split}
u_\varepsilon&(r) = -\frac{1}{8\pi} \int\limits_{\varphi=0}^{2\pi} \int\limits_{\theta=0}^\pi \int\limits_{s=0}^{+\infty} \sqrt{r^2-2rs \cos\theta + s^2} \text{e}^{-\varepsilon s^2} \text{e}^{u_\varepsilon(s)} s^2 \sin\theta \dif s \dif \theta \dif \varphi \\
    &= -\frac{1}{12r} \int_0^{+\infty} \left.(r^2 - 2rs \cos\theta + s^2)^{\frac{3}{2}}\right|_{\theta=0}^{\theta=\pi} \text{e}^{-\varepsilon s^2} \text{e}^{u_\varepsilon(s)} s \dif s \\
    &= -\frac{1}{12r} \int_0^{+\infty} \left[ (r+s)^3 - |r-s|^3 \right] \text{e}^{-\varepsilon s^2} \text{e}^{u_\varepsilon(s)} s \dif s \\
    &= -\frac{1}{6r} \int_0^r s^2(3r^2+s^2) \text{e}^{-\varepsilon s^2} \text{e}^{u_\varepsilon(s)} \dif s - \frac{1}{6} \int_r^{+\infty} s(r^2+3s^2) \text{e}^{-\varepsilon s^2} \text{e}^{u_\varepsilon(s)} \dif s.
\end{split}
\]
In the previous computation we set $y=(s \sin\theta \cos\varphi, s \sin\theta \sin\varphi, s \cos\theta)$ and we chose $x=(r,0,0)$ (recall that we have already checked the radial invariance). Now take a derivative in $r$:
\[
u'_\varepsilon(r) = \int_0^r \underbrace{\frac{s^2 - 3r^2}{6r^2}}_{<0} s^2 \text{e}^{-\varepsilon s^2} \text{e}^{u_\varepsilon(s)} \dif s - \frac{r}{3} \int_r^{+\infty} s \text{e}^{-\varepsilon s^2} \text{e}^{u_\varepsilon(s)} < 0.
\]
Hence $u_\varepsilon$ is monotone-decreasing for all $\varepsilon>0$.

By Lemma \ref{lemma:Pohozaev} one has
\[
\int_{\mathbb{R}^3} \left( u_\varepsilon(x) + 6 - 4\varepsilon|x|^2 \right) \text{e}^{-\varepsilon |x|^2} \text{e}^{u_\varepsilon(x)} \dif x = 0.
\]
Hence, since $u_\varepsilon$ is monotone-decreasing and continuous, we must have $u_\varepsilon(0) >-6$ (otherwise the previous integral would be strictly negative). Hence $-6 < u_\varepsilon(0) <0$: applying that to \eqref{eq:regularizedIntegral3D} we get
\[
\left|\int_{\mathbb{R}^3} |y| \text{e}^{-\varepsilon |y|^2} \text{e}^{u_\varepsilon(y)} \dif y \right| \leq 6
\]
and thus
\begin{equation} \label{eq:laplaceUin0}
\left| \Delta u_\varepsilon(0) \right| = \frac{1}{4\pi} \left|\int_{\mathbb{R}^3} \frac{1}{|y|} \text{e}^{-\varepsilon |y|^2} \text{e}^{u_\varepsilon(y)} \dif y \right| \leq C < +\infty.
\end{equation}
By Green's formula indeed
\[
\Delta u_\varepsilon(x) = -\frac{1}{4\pi} \int_{\mathbb{R}^3} \frac{1}{|x-y|} \text{e}^{-\varepsilon |y|^2} \text{e}^{u_\varepsilon(y)} \dif y.
\]
We now check that $\Delta u_\varepsilon$ is monotone increasing for each $\varepsilon>0$. Indeed, using again polar coordinates as before,
\[
\begin{split}
(\Delta u_\varepsilon)(r) &= -\frac{1}{4\pi} \int\limits_{\varphi=0}^{2\pi} \int\limits_{\theta=0}^\pi \int\limits_{s=0}^{+\infty} \frac{\text{e}^{-\varepsilon s^2} \text{e}^{u_\varepsilon(s)} s^2}{\sqrt{r^2 - 2rs\cos\theta + s^2}} \sin\theta \dif s \dif \theta \dif \varphi \\
    &=-\frac{1}{2r} \int_0^{+\infty} \left.\sqrt{r^2 -2rs \cos\theta + s^2}\right|_{\theta=0}^{\theta=\pi} s\, \text{e}^{-\varepsilon s^2} \text{e}^{u_\varepsilon(s)} \dif s \\
    &= -\frac{1}{2r} \int_0^{+\infty} \left[(r+s) - |r-s| \right] s \, \text{e}^{-\varepsilon s^2} \text{e}^{u_\varepsilon(s)} \dif s \\
    &= -\frac{1}{r} \int_0^r s^2 \text{e}^{-\varepsilon s^2} \text{e}^{u_\varepsilon(s)} \dif s - \int_r^{+\infty} s\, \text{e}^{-\varepsilon s^2} \text{e}^{u_\varepsilon(s)} \dif s,
\end{split}
\]
one sees that the derivative in $r$ is positive:
\[
(\Delta u_\varepsilon)'(r) = \frac{1}{r^2} \int_0^r s^2 \text{e}^{-\varepsilon s^2} \text{e}^{u_\varepsilon(s)} \dif s > 0.
\]
Now, by monotonicity of $\Delta u_\varepsilon$ and because $\Delta u_\varepsilon <0$ and \eqref{eq:laplaceUin0} hold, we have $\norm{\Delta u_\varepsilon}_{L^\infty(\mathbb{R}^3)} \leq C < \infty$.
Therefore $u_\varepsilon$ goes to some radial $u$ in $C^4_{\text{loc}}(\mathbb{R}^3)$, because of elliptic estimates.

At this point it suffices to check that there exists some $\delta > 0$, independent on $\varepsilon$, such that $u_\varepsilon(x) \leq \delta(1-|x|)$ for all $\varepsilon>0$. Indeed, that shows that the limit grows at most linearly and that
\[
u(x) = -\frac{1}{8\pi} \int_{\mathbb{R}^3} |x-y| \text{e}^{u(y)} \dif y.
\]
In fact,
\[
\left| |x-y| \text{e}^{-\varepsilon|y|^2} \text{e}^{u_\varepsilon(y)} \right| \leq |x-y| \text{e}^{\delta(1-|y|)} \in L^1(\mathbb{R}^3),
\]
so that by Lebesgue's Dominated Convergence Theorem
\[
\begin{split}
u(x) &= \lim_{\varepsilon \rightarrow 0} u_\varepsilon(x) = -\frac{1}{8\pi} \lim_{\varepsilon \rightarrow 0} \int_{\mathbb{R}^3} |x-y| \text{e}^{-\varepsilon|y|^2} \text{e}^{u_\varepsilon(y)} \dif y \\
     &= -\frac{1}{8\pi} \int_{\mathbb{R}^3} |x-y| \text{e}^{u_\varepsilon(y)} \dif y.
\end{split}
\]

Let us check then that such a $\delta>0$ exists. Observe preliminarly that $|u_\varepsilon(x)| \leq \norm{u_\varepsilon}(1+|x|)$ and $u_\varepsilon(x)<0$ for all $x \in \mathbb{R}^3$ imply that $u_\varepsilon(x) \geq -\norm{u_\varepsilon}(1+|x|)$ for all $x \in \mathbb{R}^3$. Therefore
\[
\begin{split}
-\norm{u_\varepsilon}(1+|x|) &\leq u_\varepsilon(x) = -\frac{1}{8\pi} \int_{\mathbb{R}^3} |x-y| \text{e}^{-\varepsilon |y|^2} \text{e}^{u_\varepsilon(y)} \dif y\\
    &\leq -\frac{1}{8\pi} \int_{|y|<1} |x-y| \text{e}^{-\varepsilon |y|^2} \text{e}^{u_\varepsilon(y)} \dif y \\
    &\leq -\frac{1}{8\pi} \int_{|y|<1} |x-y| \text{e}^{-\varepsilon |y|^2} \text{e}^{-\,\norm{u_\varepsilon}(1+|y|)} \dif y\\
    &\leq -\frac{1}{8\pi} \left( \int_{|y|<1} |x-y| \dif y \right) \text{e}^{-2\,\norm{u_\varepsilon} -1}.
\end{split}
\]
Now, if for the sake of contradiction we suppose that $\norm{u_\varepsilon}$ goes to zero, on the left hand side we would have something going pointwise to zero, while on the right hand side we would have something going poinwise to some strictly negative function of $x$, which is a contradiction. Hence there exists some $C>0$ such that $\norm{u_\varepsilon} \leq C$ for all $\varepsilon>0$. Thus
\[
u_\varepsilon(x) \leq -\frac{C}{8\pi} \int_{|y|<1} |x-y| \dif y \leq \delta(1-|x|),
\]
for some $\delta>0$. This completes the proof.
\end{proof}

\begin{remark}
Observe that, if $u_1(x)$ is a solution of $\Delta^2 u = \text{e}^u$, then
\[
u_\mu(x) = u_1(\mu x) + 4\log\mu
\]
is a solution as well. Therefore, actually, we have shown the existence of a whole family of trivial solutions. Once a trivial solution with the aforementioned properties $u_1$ is fixed, $u_\lambda$ can be characterized equivalently by its volume $\int_{\mathbb{R}^3} \text{e}^{u_\lambda}$, its value in 0 or its asymptotic behavior.
\end{remark}

\subsection{Non-trivial solutions}\label{subsec:non-trivial}
Now that we have a family of trivial solutions we can start looking at bifurcations. Similarly to what we did in dimension 2, we will restrict our problem to the strip $S_\lambda := \mathbb{R}^3 \times (0,\lambda)$ and find the values of $\lambda$ for which the solution is not unique. For these values of $\lambda$, we will have then non-trivial solutions in the strip $S_\lambda$. Extending them to the whole plane $\mathbb{R}^4$ by reflection and using elliptic regularity, this will give a non-trivial solution with infinite volume.

Recall that, in order to prove non-uniqueness, it suffices to find a bifurcation point $\lambda^*$ and that, by Theorem \ref{thm:krasnoselskii}, we need to find a value of the parameter for which the index of the operator associated to the differential equation changes.

To start, we need to fix the spaces of functions we are working in. Write $u(x_1,x_2,x_3,x_4) = u(x,x_4)$ (i.e., $x = (x_1,x_2,x_3)$) and define
\[
X_\lambda := \left\{ u \in C^{4,\alpha}(S_\lambda) \, \middle| \begin{aligned} &\, \frac{\partial}{\partial x_4} u(x,0) = \frac{\partial}{\partial x_4} u(x,\lambda) = 0 \quad \forall x \in \mathbb{R}^3, \\
    &\, \frac{\partial^3}{\partial x_4^3} u(x,0) = \frac{\partial^3}{\partial x_4^3} u(x,\lambda) = 0 \quad \forall x \in \mathbb{R}^3, \\
    &u \text{ radial in } x, \\
    &\left|\langle x \rangle^{\frac{1}{2}} u\right|_{4,\alpha,S_\lambda} + \left|\langle x \rangle^{\frac{5}{2}} \Delta u\right|_{2,\alpha,S_\lambda} \\
    &\qquad+ \left|\langle x \rangle^{\frac{9}{2}} \Delta^2 u\right|_{0,\alpha,S_\lambda} < +\infty \end{aligned} \right\},
\]
where $\langle x \rangle := \sqrt{1+x^2}$. Define also
\[
Y_\lambda := \left\{f \in C^{0,\alpha}(S_\lambda) \middle| \begin{aligned} \,
    &f \text{ radial in } x, \\
    &\left|\langle x \rangle^{\frac{9}{2}} f\right|_{0,\alpha,S_\lambda} < +\infty \end{aligned}
\right\}.
\]
Similarly to dimension 2, it can be checked that both $X_\lambda$ and $Y_\lambda$ are Banach spaces when endowed, respectively, with the norms
\[
\norm{u}_{X_\lambda} := \left|\langle x \rangle^{\frac{1}{2}} u\right|_{4,\alpha,S_\lambda} + \left|\langle x \rangle^{\frac{5}{2}} \Delta u\right|_{2,\alpha,S_\lambda} + \left|\langle x \rangle^{\frac{9}{2}} \Delta^2 u\right|_{0,\alpha,S_\lambda}
\]
and
\[
\norm{f}_{Y_\lambda} := \left|\langle x \rangle^{\frac{9}{2}} f\right|_{0,\alpha,S_\lambda}.
\]
Observe moreover that the functions in $X_\lambda$ grow at most like $|x|^{-\frac{1}{2}}$, while those in $Y_\lambda$ grow at most like $|x|^{-\frac{9}{2}}$, and thus $X_\lambda, Y_\lambda \subset L^2(S_\lambda)$. Notice also that perturbing the trivial solution with functions in $X_\lambda$ preserves the growth at infinity\footnote{Here the essential hypothesis is that the functions go to zero at infinity. Therefore, the choice of the power $-\frac{1}{2}$ is quite arbitrary and could be replaced by any power $\varepsilon < 0$.}.

Our problem is then finding zeros of the following functional:
\[
\begin{split}
F_\lambda : X_\lambda &\longrightarrow Y_\lambda \\
    u &\longmapsto \Delta^2 (u_0 + u) - \text{e}^{u_0+u} = \Delta^2 u - \text{e}^{u_0} (\text{e}^u - 1).
\end{split}
\]

At this point one should notice that the equation we get in this way is not in the form of Krasnosel'skii's Theorem \ref{thm:krasnoselskii}. Indeed, the operator is not in the required form $I - K$, with $K$ compact. Nonetheless, one can overcome this obstacle as follows. Suppose that we can invert the operator $\Delta^2: X_\lambda \rightarrow Y_\lambda$. Then, instead of
\begin{equation}\label{eq:4Dliouville_recentered}
\Delta^2 u - \text{e}^{u_0} (\text{e}^u - 1) = 0,
\end{equation}
one could consider the equation
\begin{equation}\label{eq:4Dliouville_inverted}
  u - \Delta^{-2}(\text{e}^{u_0} (\text{e}^u - 1)) = 0.
\end{equation}
Notice that $u$ is a solution of the original equation if and only if it is a solution of this second equation (because we are assuming that $\Delta^2$ is invertible). Hence, instead of solving \eqref{eq:4Dliouville_recentered}, we will deal with \eqref{eq:4Dliouville_inverted}. The advantage is that now one can show that \eqref{eq:4Dliouville_inverted} is in the required form, namely $\Delta^{-2} \circ F_\lambda: X_\lambda \rightarrow X_\lambda$ is a compact perturbation of the identity.

\subsubsection{Invertibility of bi-Laplacian and compactness}\label{subsubsec:invertibilityLaplacian}
%
\begin{lemma}
$\Delta^2: X_\lambda \rightarrow Y_\lambda$ is invertible. Consequently, $\Delta^{-2} \circ L : X_\lambda \rightarrow X_\lambda$ is well defined.
\end{lemma}
\begin{proof}
Let \[
Z_\lambda := \left\{w \in C^{2,\alpha}(S_\lambda) \middle| \begin{aligned} \, &\frac{\partial}{\partial x_4} w(x,0) = \frac{\partial}{\partial x_4} w(x,\lambda) = 0 \quad \forall x \in \mathbb{R}^3, \\
    &w \text{ radial in } x, \\
    &\left|\langle x \rangle^{\frac{9}{2}} w\right|_{2,\alpha,S_\lambda} + \left|\langle x \rangle^{\frac{5}{2}} \Delta w\right|_{0,\alpha,S_\lambda} < +\infty \end{aligned}
\right\}.
\]
Notice that $\Delta$ maps $X_\lambda$ to $Z_\lambda$ and maps $Z_\lambda$ to $Y_\lambda$, by construction. Therefore, it will be enough to prove that (with a slight abuse of notation) both $\Delta: X_\lambda \rightarrow Z_\lambda$ and $\Delta: Z_\lambda \rightarrow Y_\lambda$ are invertible.

Similarly to $X_\lambda$ and $Y_\lambda$, $Z_\lambda$ is a Banach space contained in $L^2(S_\lambda)$. Consider first $\Delta: X_\lambda \rightarrow Z_\lambda$: it is a linear symmetric operator (with respect to the $L^2$ product), so that it suffices to show that $\ker \Delta \subseteq X_\lambda$ is trivial. Indeed, this will imply that $\Delta$ is injective and surjective, because the operator is symmetric and thus both the image and the cokernel are contained in the domain, by elliptic regularity. But then $\Delta$ will be a bijective map between the Banach spaces $X_\lambda$ and $Z_\lambda$ which is continuous by construction, so that it will have a continuous inverse by the Open Mapping Theorem \cite[Corollary 2.7]{Brezis}.

To begin with, suppose that $u(x,x_4) = u_k(x) \cos\left(\frac{k\pi}{\lambda}x_4\right) \in \ker \Delta$, where $x$ stands for $(x_1,x_2,x_3)$. Write the Laplacian as $\Delta = \Delta_x + \frac{\partial^2}{\partial x_4^2}$, where $\Delta_x$ is the Laplacian in the first three coordinates only. Then one gets
\[
\begin{split}
0 &= \Delta \left(u_k(x) \cos\left(\frac{k\pi}{\lambda}x_4\right)\right) \\
  &= \left( \Delta_x u_k(x) - 2\frac{k^2 \pi^2}{\lambda^2} u_k(x) \right) \cos\left(\frac{k\pi}{\lambda}x_4\right),
\end{split}
\]
so that $u \in \ker \Delta$ if and only if $\left( \Delta_x - \frac{k^2 \pi^2}{\lambda^2} \right) u_k(x) = 0$. Now, if $k=0$, then the equation becomes $\Delta_x u_0(x) = 0$. By the Maximum Principle and the growth requirements at infinity, we must then have $u_0=0$.

Let now $k \neq 0$, so that we have the equation $\Delta_x u_k(x) = \frac{k^2 \pi^2}{\lambda^2} u_k(x)$. Suppose that $u_k$ attains a maximum at $x = x_M \in \mathbb{R}^3$, then $\Delta_x u_k(x_M) \leq 0$ implies that $u_k(x) \leq u_k(x_M) \leq 0$ for all $x \in \mathbb{R}^3$. If instead $u_k$ has no interior maximum, then it must be at infinity, and thus $u_k(x) \leq 0$ for all $x \in \mathbb{R}^3$, because according to our choice of $X_\lambda$ we have $u_k(x) \rightarrow 0$ as $|x| \rightarrow +\infty$. Similar considerations with the minimum of $u_k$ lead to $u_k(x) \geq 0$ for all $x \in \mathbb{R}^3$. But then $u_k = 0$ also for $k \neq 0$.

Given now any $u(x,x_4)$ in the kernel of the Laplacian, the previous argument shows then that all coefficients of its Fourier expansion in $x_4$ are forced to be zero. Therefore, $u=0$ and thus $\ker \Delta = 0$, proving that $\Delta: X_\lambda \rightarrow Z_\lambda$ is invertible with continuous inverse.

The same argument proves that $\Delta: Z_\lambda \rightarrow Y_\lambda$ is invertible with continuous inverse, so that $\Delta^2 = \Delta \circ \Delta: X_\lambda \rightarrow Y_\lambda$ is invertible with continuous inverse.
\end{proof}

As for compactness, notice first that the operator $\Delta^{-2} \circ F_\lambda$ already has the form $I - K$, with $K(u) = \Delta^{-2} (\text{e}^{u_0}(\text{e}^u  - 1))$.
\begin{lemma}
  $K: X_\lambda \rightarrow X_\lambda$ is a compact operator.
\end{lemma}
\begin{proof}
  Using the same notation of the previous Lemma, we have $\Delta^{-2} f = G * f$. Here, the important property of $G$ is that it grows slower than the exponential as $|x| \rightarrow +\infty$ (see again \cite{Cohl2013}). Then the result follows from repeated application of Arzel\`a-Ascoli's Theorem.

  Let indeed $\{u_k\}_k$ be a bounded sequence in $X_\lambda$, i.e.~$\norm{u_k}_{X_\lambda} \leq C$ for all $k \in \mathbb{N}$ for some $C < +\infty$ independent on $k$. Then, for any $x \in S_\lambda$,
  \[
  |K(u_k)(x)| = \left| \int_{S_\lambda} G(x-y) \text{e}^{u_0(y)} (\text{e}^{u(y)}  - 1) \dif y \right| \leq C_x,
  \]
  because $\text{e}^{u_0(y)}$ goes to zero as $\text{e}^{-|y|}$ for $|y| \rightarrow +\infty$ (so, it goes to zero much faster than how all other terms go to infinity), making the integral converge for all $x \in S_\lambda$. Therefore, $\{K(u_k)\}_k$ is uniformly bounded. Similarly, for all $x \in S_\lambda$, one has
  \[
  \left| \nabla K(u_k)(x) \right| \leq \int_{S_\lambda} |\nabla G(x-y)| \text{e}^{u_0(y)} \left|\text{e}^{u(y)}  - 1\right| \dif y \leq C'_x,
  \]
  again because $\text{e}^{u_0(y)}$ goes to zero as $\text{e}^{-|y|}$ and all the other terms do not grow exponentially to infinity. Hence, $\{K(u_k)\}_k$ is also equicontinuous and thus, by Arzel\`a-Ascoli, it converges up to subsequences in the $C^0$ norm.

  The same argument applied to the derivatives of $K(u_k)$ gives convergence in all $C^j$ norms (because we can make the derivative fall on $G$ and apply the same argument as before), and thus in particular in the $C^{4,\alpha}$ norm. Moreover, we have already seen that all the $K(u_k)$'s satisfy the growth requirements, so that $\{K(u_k)\}_k$ has a converging subsequence in the $X_\lambda$ norm, proving compactness of $K$.
\end{proof}

\begin{remark}
  The same argument proves also that the linearization $\Delta^{-2} \circ L$ has the form $I-K$ with $K$ a compact operator.
\end{remark}

\subsubsection{Change of index and bifurcation}
Now that the setup of our problem is complete, we can move to actually proving bifurcation. As explained at the beginning of this Section, we plan to use Krasnosel'skii's Bifurcation Theorem \ref{thm:krasnoselskii}, so that we need to prove that the index of the operator $\Delta^{-2} \circ F_\lambda$ changes for some $\lambda$, i.e.~that the dimension of the negative space of its linearization $\Delta^{-2} \circ L_\lambda$ changes for some $\lambda$.

\begin{lemma}
The linearized operator
\[
L_\lambda[v] := \Delta^2 v - \text{\emph{e}}^{u_0} v
\]
admits an eigenfunction not depending on $x_4$ and with negative eigenvalue.
\end{lemma}
\begin{proof}
Observe that $L_\lambda$ can be restricted to an operator
\[
\tilde{L}_\lambda := L_\lambda|_{X_\lambda^0}: X_\lambda^0 \rightarrow Y_\lambda,
\]
where $X_\lambda^0$ is the subset of $X_\lambda$ of functions not depending on $x_4$. The elements of $X_\lambda^0$ are then actually functions of $\mathbb{R}^3$, so in the rest of this proof we will just drop the dependence on $x_4$.

We construct a function $v: \mathbb{R}^3 \rightarrow \mathbb{R}$ that is compactly supported, radial and such that $\langle \tilde{L}_\lambda v, v \rangle_{L^2}<0$.

Define
\[
f(r) := \begin{cases}
0 &\text{ if } r \leq 1 \\
\text{e}^{-\frac{1}{(r-1)^2}} \text{e}^{-\frac{1}{(r-2)^2}} &\text{ if } 1 < r < 2 \\
0 &\text{ if } r \geq 2
\end{cases}
\]
and take
\[
v(x_1,x_2,x_3) := \frac{1}{A} \int_{\sqrt{x_1^2+x_2^2+x_3^2}}^{+\infty} f(s) \dif s,
\]
with
\[
A := \int_0^{+\infty} f(s) \dif s.
\]
Notice that $v$ is bounded with compact support, so that it belongs to the space of functions $X_\lambda^0$. Observe also that $v$ is constantly equal to 1 if $r:=\sqrt{x_1^2+x_2^2+x_3^2}<1$ and is identically 0 outside the ball $B^3(0,2)$, so that its Laplacian is different from zero only in the annulus $1 \leq r \leq 2$. Therefore, setting
\[
V\!\left(\sqrt{x_1^2+x_2^2+x_3^2}\right):=v(x_1,x_2,x_3),
\]
one gets
\[
\int_{S_\lambda} (\Delta v)^2 \dif x = \lambda \int_1^2 \left( V''(r) + \frac{2}{r} V'(r) \right)^{\!2} 4\pi r^2 \dif r = C < +\infty.
\]
Fix now a trivial solution $u_1$, as found in Subsection \ref{subsec:trivial}. Recall that we thus have the family $\{u_\mu\}_\mu$ of trivial solutions (as the functions in this family do not depend on $x_4$, they can be tought as a functions on $\mathbb{R}^3$). We will then show that we can choose $u_0 \in \{u_\mu\}_\mu$ so that
\[
\int_{S_\lambda} \text{e}^{u_0(x)} [v(x)]^2 \dif x
\]
is sufficiently large. In fact
\[
\begin{split}
\int_{S_\lambda} &\text{e}^{u_\mu(x)} [v(x)]^2 \dif x = \lambda \int_{\mathbb{R}^3} \mu^4 \text{e}^{u_1(\mu x)} \dif x \\
    &= \lambda \int_{\mathbb{R}^3} \mu^4 \text{e}^{u_1(y)} v^2\left(\frac{y}{\mu}\right) \frac{\dif y}{\mu^3} \geq \lambda \mu \int_{|y|\leq\mu} \text{e}^{u_1(y)}\underbrace{v^2\left(\frac{y}{\mu}\right)}_{1} \dif y \\
    &= \lambda\mu \int_{|y|\leq\mu} \text{e}^{u_1(y)} \dif y \underset{\mu \rightarrow +\infty}\longrightarrow +\infty
\end{split}
\]
Summing up, if we fix $u_0 := u_\mu$ with $\mu$ sufficiently large, then $\langle \tilde{L}_\lambda v,v \rangle_{L^2} < 0$, with $v$ the function defined before.

Now, as $X_\lambda^0 \subset L^2(\mathbb{R}^3)$, we get that $\tilde{L}_\lambda$ is self-adjoint and semibounded (in the $L^2$ norm):
\[
\begin{split}
\langle \tilde{L}_\lambda v, v \rangle_{L^2} &= \int_{\mathbb{R}^3} \left[\left(\Delta^2 v(x)\right) v(x) - \text{e}^{u_0(x)} v^2(x)\right] \dif x \\
    &= \int_{\mathbb{R}^3} \left[\left(\Delta v(x)\right)^2 - \text{e}^{u_0(x)} v^2(x)\right] \dif x \\
    &\geq -\int_{\mathbb{R}^3} \text{e}^{u_0(x)} v^2(x) \dif x \geq -\text{e}^{u_0(0)} \norm{v}_{L^2(\mathbb{R}^3)}.
\end{split}
\]
Hence we have that the lowest eigenvalue for $\tilde{L}_\lambda$ with eigenfunction in $X_\lambda^0$ (note that $v \in X_\lambda^0$ and recall that $X_\lambda^0$ is a Banach space) is
\[
\nu_0 = \min_{u \in X_\lambda^0 \setminus \{0\}} \frac{\langle \tilde{L}_\lambda u, u \rangle_{L^2}}{\norm{u}_{L^2}^2} \leq \frac{\langle \tilde{L}_\lambda v, v \rangle_{L^2}}{\norm{v}_{L^2}^2} < 0
\]
(see for example \cite[Theorem 11.4]{Helffer2013}). Notice that $\nu_0$ is an eigenvalue of $L_\lambda: X_\lambda \rightarrow Y_\lambda$ as well (possibly not the first one) because $X_\lambda^0 \subset X_\lambda$ and that a corresponding eigenfunction $v_0 \in X_\lambda^0$ is an eigenfunction of $L_\lambda$ (extending it trivially in $x_4$). Note that $v_0$ does not depend on $x_4$ and $\nu_0 < 0$, as required.
\end{proof}

This negative eigenfunction $v_0$ for $L_\lambda$ (not depending on $x_4$) is negative also for $\Delta^{-2} \circ L_\lambda$ (even if not necessarily an eigefunction):
\[
\langle (\Delta^{-2} \circ L_\lambda) w_0, w_0 \rangle_{L^2} = \mu_0 \langle \Delta^{-2} w_0, w_0 \rangle_{L^2} = \mu_0 \norm{\Delta^{-1}w_0}_{L^2}^2 < 0,
\]
where the last equality comes from self-adjointness of $\Delta^{-1}$, while the strict inequality comes from the fact that, if $\norm{\Delta^{-1}w_0}_{L^2} = 0$, then $\norm{w_0}_{L^2}=0$, a contradiction.

\begin{remark}
  As this function $v_0$ does not depend on the last variable $x_4$, it is independent on the choice of $\lambda$. Therefore, in what follows, even if the family of functions that we construct depends on $\lambda$, the way we construct it does not. To simplify our notation, we will not explicitly indicate the dependence of the family of functions on $\lambda$.
\end{remark}

We now want to prove that we have a family of linearly independent functions in the negative space of $L_\lambda$ that gets bigger as $\lambda$ get larger. For a fixed $\lambda$, consider the functions
\[
v_k(x_1,x_2,x_3,x_4) := v_0(x_1,x_2,x_3) \cos\left(\frac{2\pi k}{\lambda} x_4\right).
\]
Observe that the $v_k$'s are orthogonal, satisfy Neumann conditions on $\partial S_\lambda$ and their third normal derivatives on $\partial S_\lambda$ vanish, so that they belong to $X_\lambda$.

\begin{lemma}\label{lemma:negativeFunctions}
  The number of values of $k$ for which $\langle (\Delta^{-2} \circ L_\lambda) v_k, v_k \rangle_{L^2}<0$ is finite for all $\lambda>0$ and goes to infinity as $\lambda \rightarrow +\infty$.
\end{lemma}
\begin{proof}
Recall first that we already know that $\langle (\Delta^{-2} \circ L_\lambda) v_0, v_0 \rangle_{L^2} < 0$. We also have
\begin{equation}\label{eq:Lv_k}
\begin{split}
  \langle (\Delta^{-2} \circ L_\lambda) v_k, v_k \rangle_{L^2} &= \langle v_k - \Delta^{-2}(\text{e}^{u_0}v_k), v_k \rangle_{L^2} \\
    &= \norm{v_k}_{L^2}^2 - \langle \Delta^{-2}(\text{e}^{u_0}v_k), v_k \rangle_{L^2} \\
    &= \norm{v_0}_{L^2}^2 - \langle \text{e}^{u_0} v_k, \Delta^{-2} v_k \rangle_{L^2},
\end{split}
\end{equation}
where $\norm{v_k}_{L^2} = \norm{v_0}_{L^2}$ because $\int_0^\lambda \cos^2\left(\frac{2\pi k}{\lambda}x_4\right) \dif x_4 = 1$.

We now have to compute $\Delta^{-2} v_k$ explicitly. Set $w = \Delta^{-2} v_k$, i.e.~$w$ is the solution of $\Delta^2 w = v_k$. Decomposing $w$ into its Fourier modes in the $x_4$ variable we get
\[
w(x,x_4) = \sum_{n=0}^{+\infty} w_n(x) \cos\left( \frac{2\pi n}{\lambda} x_4 \right)
\]
and thus, taking $\Delta^2$ of both sides,
\[
v_0(x) \cos\left( \frac{2\pi k}{\lambda} x_4 \right) = \sum_{n=0}^{+\infty} \left[\left(\Delta_x - \frac{2\pi n}{\lambda}\right)^2 w_n(x) \right] \cos\left(\frac{2\pi n}{\lambda} x_4\right).
\]
For the same Maximum Principle argument of Subsection \ref{subsubsec:invertibilityLaplacian}, the operators $\left(\Delta_x - \frac{2\pi n}{\lambda}\right)^2$ are invertible, and so we get
\[
\Delta^{-2} v_k(x,x_4) = \left[\left(\Delta_x - \frac{2\pi k}{\lambda}\right)^{-2} v_0(x) \right] \cos\left(\frac{2\pi k}{\lambda} x_4\right).
\]
Plugging this into \eqref{eq:Lv_k}, we then obtain
\[
\langle (\Delta^{-2} \circ L_\lambda) v_k, v_k \rangle_{L^2} = \norm{v_0}_{L^2}^2 - \left\langle \text{e}^{u_0} v_0, \left(\Delta_x - \frac{2\pi k}{\lambda}\right)^{-2} v_0 \right\rangle_{L^2}.
\]
Consider now the function $f: \mathbb{R}_{>0} \rightarrow \mathbb{R}$ defined as
\[
f(t) := \norm{v_0}_{L^2}^2 - \left\langle \text{e}^{u_0} v_0, (\Delta_x - t)^{-2} v_0 \right\rangle_{L^2}.
\]
This is a continuous map, because the map that takes an invertible operator to its inverse is continuous in the operator norm \cite[Prop. 2.1.1]{AmbrosettiProdi}. By construction
\[
f(0) = \norm{v_0}_{L^2}^2 - \left\langle \text{e}^{u_0} v_0, \Delta^{-2} v_0 \right\rangle_{L^2} = \langle (\Delta^{-2} \circ L_\lambda) v_0,v_0 \rangle_{L^2} < 0,
\]
so that $f(t) < 0$ for all $t < \varepsilon$ for some $\varepsilon > 0$. Notice that the number of $k$'s for which $t = \frac{2\pi k}{\lambda} < \varepsilon$ grows with $\lambda$. Therefore, to conclude that the number of $k$'s such that $\langle (\Delta^{-2} \circ L_\lambda) v_k, v_k \rangle_{L^2} < 0$ grows with $\lambda$, it suffices to prove that for all fixed $\lambda>0$ there are only finitely many $k$'s for which $\langle (\Delta^{-2} \circ L_\lambda) v_k, v_k \rangle_{L^2} < 0$. This is equivalent to saying that there exists $M > 0$ such that $f(t) > 0$ for all $t > M$. Specifically, we are going to prove that $f(t) \rightarrow \norm{v_0}^2>0$ as $t \rightarrow +\infty$.

Fix indeed $u \in X_\lambda$ independent on $x_4$, then $\norm{(\Delta_x - t)^2 u}_{C^0} \sim |u|_{C^0} t^2$ as $t \rightarrow +\infty$ (recall that $|\Delta_x^2 u|_{C^0},|\Delta_x u|_{C^0} < +\infty$ are fixed). Then
\[
\norm{(\Delta_x - t)^{-2} v_0}_{C^0} \rightarrow 0, \quad t \rightarrow +\infty,
\]
decreasing like $t^{-2}$. Recalling that $\text{e}^{u_0} v_0 \in L^1$, we conclude then that
\[
\left| \left\langle \text{e}^{u_0} v_0, (\Delta_x - t)^{-2} v_0 \right\rangle_{L^2} \right| \leq \norm{\text{e}^{u_0} v_0}_{L^1} \norm{(\Delta_x - t)^{-2} v_0}_{C^0} \rightarrow 0
\]
as $t \rightarrow +\infty$. Thus $f(t) \rightarrow \norm{v_0}_{L^2}^2$ when $t \rightarrow +\infty$, as wanted.
\end{proof}

\begin{lemma}\label{lemma:negativeSpace}
  The dimension of the negative space of $\Delta^{-2} \circ L_\lambda$ goes to $+\infty$ as $\lambda \to +\infty$.
\end{lemma}
\begin{proof}
Fix $\lambda$ and suppose that $v_0, \dots, v_r$ are the negative functions for $L_\lambda$ of Lemma \ref{lemma:negativeFunctions}.

Notice first that, if $v_i,v_j$ are negative for $\Delta^{-2} \circ L_\lambda$, then any their linear combination $\alpha v_i + \beta v_j$ is negative as well. Indeed: if $i=j$ this is trivial, while if $i \neq j$ then
\[
\begin{split}
  \langle \Delta^{-2} &\circ L_\lambda (\alpha v_i + \beta v_j), \alpha v_i + \beta v_j \rangle_{L^2} \\
    &= \alpha^2 \langle (\Delta^{-2} \circ L_\lambda) v_i, v_i \rangle_{L^2} + \beta^2 \langle (\Delta^{-2} \circ L_\lambda) v_j, v_j \rangle_{L^2} \\
    &\quad + \alpha\beta \langle (\Delta^{-2} \circ L_\lambda) v_i, v_j \rangle_{L^2} + \alpha\beta \langle (\Delta^{-2} \circ L_\lambda) v_j, v_i \rangle_{L^2}
\end{split}
\]
and
\[
\begin{split}
  \langle &(\Delta^{-2} \circ L_\lambda) v_i, v_j \rangle_{L^2} \\
    &= \langle v_i, v_j \rangle_{L^2} \\
    &\quad - \left\langle \text{e}^{u_0} v_0(x) \cos\left(\frac{2\pi i}{\lambda} x_4 \right), \left[\left( \Delta_x - \frac{2\pi j}{\lambda} \right)^{-2} v_0(x) \right] \cos\left(\frac{2\pi j}{\lambda}\right) \right\rangle_{L^2} = 0,
\end{split}
\]
because $\int_0^\lambda \cos\left(\frac{2\pi i}{\lambda} x_4\right) \cos\left(\frac{2\pi j}{\lambda} x_4\right) \dif x_4 = 0$ for $i \neq j$, so that
\[
\begin{split}
  \langle \Delta^{-2} &\circ L_\lambda (\alpha v_i + \beta v_j), \alpha v_i + \beta v_j \rangle_{L^2} \\
    &= \alpha^2 \langle (\Delta^{-2} \circ L_\lambda) v_i, v_i \rangle_{L^2} + \beta^2 \langle (\Delta^{-2} \circ L_\lambda) v_j, v_j \rangle_{L^2} < 0.
\end{split}
\]

Then, by Rayleigh's Min-Max Principle \cite[Theorem 12.1]{LiebLoss2001} the first $r$ eigenvalues of $L_\lambda$ are negative: for $n \leq r$ we have
\[
\begin{split}
  \nu_n &= \min_{\varphi_1,\dots,\varphi_n} \max \{ \langle L_\lambda \varphi, \varphi \rangle_{L^2} \mid \varphi \in \text{span}(\varphi_1,\dots,\varphi_n), \,\, \norm{\varphi}_{L^2}=1 \} \\
    &\leq \max \{ \langle L_\lambda \varphi, \varphi \rangle_{L^2} \mid \varphi \in \text{span}(v_1,\dots,v_n), \,\, \norm{\varphi}_{L^2}=1 \} < 0.
\end{split}
\]
As for each of this eigenvalues $\nu_k$ (counted with multiplicity) there is an eigenfunction linearly independent to the previous ones, we conclude that the dimension of the negative space of $\Delta^{-2} \circ L_\lambda$ must grow with $\lambda$. Indeed, Lemma \ref{lemma:negativeFunctions} tells us that the number of elements in the negative family of functions grows with $\lambda$, and thus, as $\lambda$ increases, Rayleigh's Min-Max Principle gives us more negative eigenvalues.
\end{proof}

\begin{proof}[Proof of Theorem \ref{thm:4Dnon-trivialSolution}]
Lemma \ref{lemma:negativeSpace}, along with the fact that $\Delta^{-2} \circ L = I - K$ (where $I$ is the identity and $K$ is a compact operator), shows that there must be some value $\lambda^*$ of the parameter $\lambda$ for which the number of negative eigenvalues changes. Indeed, notice that $f$ is an eigenfunction of $-K$ with eigenvalue $\mu$ if and only if $f$ is an eigenfunction of $I-K$ with eigenvalue $1+\mu$. Recalling that the spectrum of the compact operator $-K$ is bounded and accumulating at most in 0 \cite[Theorem 6.16]{Helffer2013}, we then get that the spectrum of $I-K$ is bounded and accumulating at most in 1. This implies that there are only finitely many negative eigenvalues, for each fixed value of $\lambda$. But, according to Lemma \ref{lemma:negativeSpace}, the number of negative eigenvalues goes to infinity as $\lambda \rightarrow +\infty$. Consequently, there must be some $\lambda^*$ for which the number of negative eigenvalues changes. But then at $\lambda^*$ the dimension of the space of negative eigenfunctions grows, so that $\lambda^*$ is a point of changing index for the operator $\Delta^{-2} \circ F_\lambda$.

Now we can finally apply Krasnosel'skii's Bifurcation Theorem \ref{thm:krasnoselskii} to the equation $\Delta^{-2} \circ F_\lambda(u) = 0$. Indeed, we proved that $\Delta^{-2} \circ F_\lambda$ has the form $I - K$ with $K$ compact, and that $\lambda^*$ is a point of changing index. Hence, $\lambda^*$ is a point of bifurcation for $\Delta^{-2} \circ F_\lambda$. This means that there is a non-trivial solution of $\Delta^{-2} \circ F_\lambda(u) = 0$ for some $\lambda$ which is a perturbation of $u_0$, i.e.~a zero of $F_\lambda$ on $S_\lambda$. We can then extend $u$ to the whole $\mathbb{R}^4$ by reflecting it along the boundaries of the strip and repeating this procedure. Continuity is then assured by construction, the accordance of first and third normal derivatives comes from the definition of the space $X_\lambda$, and accordance of second derivatives is again given by our construction (the operator $\partial_{x_4}^2$ is indeed invariant under the reflection $x_4 \leftrightarrow -x_4$). Therefore, this is a weak solution and thus a strong solution, by bootstrapping and applying twice Theorem \ref{thm:weightedSchauder}. As this solution is periodic and non-zero, it must then have infinite volume $\int_{\mathbb{R}^4} \text{e}^u = +\infty$.
\end{proof}

\appendix
\section{Proof of Weighted Schauder's Estimates Theorem \ref{thm:weightedSchauder}}\label{app:weightedSchauder}
Let $u$ be a distributional solution of
\begin{equation}\label{eq:originalElliptic}
Lu = a^{ij}(x) \partial_i \partial_j u + b^i(x) \partial_i u + c(x) u = f,
\end{equation}
where $L$ is a uniformly elliptic operator. Take $w \in C^{2,\alpha}(\Omega), w > 0$ and set $v := wu$. We have
\begin{equation}\label{eq:der1}
\partial_i v(x) = \partial_i w(x) u(x) + w(x) \partial_i u(x)
\end{equation}
and
\begin{equation}\label{eq:der2}
\begin{split}
\partial_i \partial_j v(x) =\,\, &\partial_i \partial_j w(x) u(x) + \partial_i w(x) \partial_j u(x) \\
    &+ \partial_j w(x) \partial_i u(x) + w(x) \partial_i \partial_j u(x).
\end{split}
\end{equation}
Contracting \eqref{eq:der2} with $a^{ij}(x)$ and using \eqref{eq:originalElliptic} we get
\begin{equation}\label{eq:mixedElliptic}
\begin{split}
a^{ij}(x) &\partial_i \partial_j v(x) + \left(w(x) b^j(x) - 2a^{ij}(x) \partial_i w(x) \right)\partial_j u(x) \\
&+ \left(c(x)w(x) - a^{ij}(x) \partial_i \partial_j w(x) \right) u(x) = w(x)f(x) =: g(x).
\end{split}
\end{equation}
where we also use the fact that, by hypothesis, $a^{ij} = a^{ji}$.

According to the definition of $v$ and to equation \eqref{eq:der1}, we know that:
\begin{itemize}
\item $u(x) = \frac{v(x)}{w(x)}$ (recall that $w$ is never 0),
\item $\partial_i u(x) = \frac{1}{w(x)} \left(\partial_i v(x) - \frac{\partial_i w(x)}{w(x)} v(x)\right)$.
\end{itemize}
Hence we get
\begin{equation}\label{eq:weightedElliptic}
\begin{split}
&a^{ij}(x) \partial_i \partial_j v(x) + \left(b^j(x) - 2a^{ij}(x) \partial_i \log w(x) \right) \partial_j v(x) \\
&+ \! \Big(2a^{ij} \partial_i \log w(x) \partial_j \log w(x) \\
&\quad - a^{ij}(x) \frac{\partial_i \partial_j w(x)}{w(x)} - b^j(x) \partial_j \log w(x) + c(x) \Big) v(x) \\
&\qquad = w(x)f(x) =: g(x).
\end{split}
\end{equation}
Observe now that what we got in this way is still a uniformly elliptic equation, as the highest order coefficients are still the same.

Assume then that the hypothesis on the coefficients of $L$ given by Theorem \ref{thm:SchauderInterior} hold. In order to apply the Schauder's estimates on equation \eqref{eq:weightedElliptic}, then, it suffices to check that there exists some positive constant $\tilde{\Lambda}$ such that
\[
|\tilde{b}^i|^{(1)}_{0,\alpha,\Omega}, |\tilde{c}|^{(2)}_{0,\alpha,\Omega} \leq \tilde{\Lambda},
\]
where $\tilde{b}^i$ and $\tilde{c}$ are, respectively, the 1-st and 0-th order coefficients of the new equation \eqref{eq:weightedElliptic}. Moreover, we will need $g \in C^{0,\alpha}(\Omega)$, which is precisely $f \in C^{0,\alpha}_w(\Omega)$.

We first deal with the first order coefficient. By the triangular inequality we have
\[
\begin{split}
|\tilde{b}^i|^{(1)}_{0,\alpha,\Omega} &= |b^i - 2a^{ij} \partial_j \log w|^{(1)}_{0,\alpha,\Omega} \\
    &\leq |b^i|^{(1)}_{0,\alpha,\Omega} + 2|a^{ij} \partial_j \log w|^{(1)}_{0,\alpha,\Omega},
\end{split}
\]
hence, it suffices to show that $|a^{ij} \partial_j \log w|^{(1)}_{0,\alpha,\Omega}$ is finite. By hypothesis we know that $|a^{ij}|^{(0)}_{0,\alpha,\Omega} \leq \Lambda$. Consequently, by Proposition \ref{prop:CauchySchwarz}, we get that
\[
|a^{ij} \partial_j \log w|^{(1)}_{0,\alpha,\Omega} \leq |a^{ij}|^{(0)}_{0,\alpha,\Omega} |\partial_j \log w|^{(1)}_{0,\alpha,\Omega} \leq \Lambda |\partial_j \log w|^{(1)}_{0,\alpha,\Omega}.
\]
Thus, it suffices to require that
\[
|\partial_j \log w|^{(1)}_{0,\alpha,\Omega} \leq C_1 < +\infty, \quad \forall j.
\]

We now move to the 0-th order coefficient. Again by the triangular inequality
\[
\begin{split}
&|\tilde{c}|^{(2)}_{0,\alpha,\Omega} = \left|c - b^j \partial_j \log w - a^{ij} \frac{\partial_i \partial_j w}{w} + 2a^{ij} \partial_i \log w \partial_j \log w \right|^{(2)}_{0,\alpha,\Omega} \\
    &\leq |c|^{(2)}_{0,\alpha,\Omega} + |b^j \partial_j \log w|^{(2)}_{0,\alpha,\Omega} + \left|a^{ij} \frac{\partial_i \partial_j w}{w}\right|^{(2)}_{0,\alpha,\Omega} + 2|a^{ij} \partial_i \log w \partial_j \log w|^{(2)}_{0,\alpha,\Omega}.
\end{split}
\]
By hypothesis, the first summand in this last expression is bounded by $\Lambda$. Moreover, again by Proposition \ref{prop:CauchySchwarz}, we get
\[
|b^j \partial_j \log w|^{(2)}_{0,\alpha,\Omega} \leq |b^j|^{(1)}_{0,\alpha,\Omega} |\partial_j \log w|^{(1)}_{0,\alpha,\Omega} \leq \Lambda C_1,
\]
\[
|a^{ij} \partial_i \log w \partial_j \log w|^{(2)}_{0,\alpha,\Omega} \leq |a^{ij}|^{(0)}_{0,\alpha,\Omega} |\partial_i \log w|^{(1)}_{0,\alpha,\Omega} |\partial_j \log w|^{(1)}_{0,\alpha,\Omega} \leq \Lambda C_1^2
\]
and
\[
\left|a^{ij} \frac{\partial_i \partial_j w}{w}\right|^{(2)}_{0,\alpha,\Omega} \leq |a^{ij}|^{(2)}_{0,\alpha,\Omega} \left|\frac{\partial_i \partial_j w}{w}\right|^{(2)}_{0,\alpha,\Omega} \leq \Lambda \left|\frac{\partial_i \partial_j w}{w}\right|^{(2)}_{0,\alpha,\Omega},
\]
meaning that it suffices to require that
\[
\left|\frac{\partial_i \partial_j w}{w}\right|^{(2)}_{0,\alpha,\Omega} \leq C_2 < +\infty, \quad \forall i,j.
\]

\section{Computations on Legendre functions}\label{app:2Dcomputations}
We solve the equation
\[
-\tilde{v}_j''(x) - 2\sech^2(x) \tilde{v}_j(x) = -\left(\frac{\pi j}{\lambda}\right)^2 \tilde{v}_j(x), \quad \forall	x \in \mathbb{R}, \, \forall j \in \mathbb{N}.
\]
for fixed $j \in \mathbb{N}$. First, if $j=0$ this becomes
\[
\tilde{v}_0''(x) + 2\sech^2(x) \tilde{v}_0(x) = 0,
\]
which has general solution
\[
\tilde{v}_0(x) = c_1 \tanh(x) + c_2 \left(-\frac{1}{2} \tanh(x) \log \frac{1-\tanh(x)}{1+\tanh(x)} - 1 \right).
\]
Now, notice that the first summand is odd\footnote{Indeed, it is the $x$ derivative of the trivial solution $u_0$, which is even. Observe that the $x$ derivative of a trivial solution is always a solution of the linearized equation. In fact, as $u_0$ is a solution of the original equation
\[
\Delta u_0(x,y) + \text{e}^{u_0(x,y)} = 0,
\]
taking the derivative in $x$ of this expression one gets
\[
\Delta \pd{u_0}{x}(x,y) + \text{e}^{u_0(x,y)} \pd{u_0}{x}(x,y) = 0,
\]
which means that $\pd{u_0}{x}$ is a solution of the linearized problem in $u_0$.
} and the second is even. Thus, $c_1=0$. Moreover, we also have $c_2=0$, as the second summand grows linearly at infinity (i.e., faster than the requirements). Hence, we can already exclude the possibility of having elements in $\ker L$ with $j=0$.

Let then $j>0$ and make the substitution $y = \tanh(x)$:
\[
\left[ \left(1-y^2\right) \tilde{v}_j'(y) \right]' + 2 \tilde{v}_j(y) - \frac{1}{1-y^2} \left( \frac{\pi j}{\lambda} \right)^2 \tilde{v}_j(y) = 0.
\]
We get then a Legendre equation with integer degree $l = 1$ and with order $\mu = \frac{2\pi j}{\lambda}$. A general solution is then given by a linear combination of first and second order Legendre functions:
\[
\tilde{v}_j(x) = A \, P_1^{\frac{\pi j}{\lambda}}(\tanh(x)) + B \, Q_1^{\frac{\pi j}{\lambda}}(\tanh(x)).
\]
Actually, not all the values of $A$ and $B$ are admissible, as we shall immediately see. The following expansions can be found, for instance, on \cite{NIST}.

Suppose first that $B=0$ and consider thus $P_1^\frac{\pi j}{\lambda}$ only. It is known that
\[
P_1^\mu(y) \underset{y \to 1^-}{\sim} \frac{1}{\Gamma(1-\mu)} \left( \frac{2}{1-y} \right)^{\frac{\mu}{2}}
\]
for $\mu \not\in \mathbb{N}$. Therefore, for such values of $\mu$, $P_1^\mu(\tanh(x)) \sim_{x \to +\infty} C \text{e}^{\mu x}$, meaning that such a solution cannot lead to functions in the space $X_\lambda$. Hence, we know that, if $B=0$, $\mu$ must be an integer. We now recall that, if $\mu$ is an integer and $\mu > l$, then $P_l^\mu \equiv 0$. Consequently, if $B=0$, the only non-trivial solution is the one with $\mu = 1$, namely $1 = \frac{\pi j}{\lambda}$. Explicitly:
\[
\tilde{v}_j(x) = \begin{cases}
A \sech(x) & \text{if } j = \frac{\lambda}{\pi} \\
0 \quad & \text{otherwise}
\end{cases}
\]
(notice that $\sech(x)$ is even).

Suppose now that $A=0$ and consider $Q_1^\frac{\pi j}{\lambda}$ only. It is known that
\[
Q_1^\mu(y) \underset{y \to 1^-}{\sim} \frac{1}{2} \cos(\mu \pi) \Gamma(\mu) \left(\frac{2}{1-y}\right)^\frac{\mu}{2}
\]
for $\mu \not= \frac{1}{2}, \frac{3}{2}, \frac{5}{2}, \dots$. As before, then, for such values of $\mu$ we have that $Q_1^\mu(\tanh(x)) \sim_{x \to +\infty} C \text{e}^{\mu x}$. Hence, in order to have functions in $X_\lambda$, we must require that $\mu = \frac{1}{2}, \frac{3}{2}, \frac{5}{2}, \dots$. In this case the expansion at $1^-$ becomes
\[
Q_1^\mu(y) \underset{y \to 1-}{\sim} (-1)^{\mu + \frac{1}{2}} \frac{\pi \Gamma(\mu+2)}{2 \Gamma(\mu+1) \Gamma(2-\mu)} \left(\frac{1-y}{2}\right)^\frac{\mu}{2}
\]
if $1 \pm \mu = l \pm \mu \not= -1, -2, -3, \dots$ (which is trivially true). Therefore, the function stays bounded for $x \to +\infty$ (i.e.~$y \to 1^-$). Nonetheless,
\[
Q_1^\mu(y) = -\cos((1 + \mu)\pi)Q_1^\mu(-y) - \frac{\pi}{2} \sin((1+\mu)\pi)P_1^\mu(-y)
\]
immediately shows that the function blows-up as $y \to -1^+$, i.e. as $x \to -\infty$. In this way we have excluded all the possible $\mu$'s and we can therefore assess that, in order to have a non-trivial $v_j \in X_\lambda$, it must be $A \not= 0$.

We finally have to check that there are no combinations of $A, B \not= 0$ that lead to solutions in $X_\lambda$. Observe first that, according to what we said before
\begin{itemize}
\item $\mu \in \mathbb{N}$ implies that $A \, P_1^\mu(\tanh(x)) + B \, Q_1^\mu(\tanh(x))$ blows-up exponentially at both $+\infty$ and $-\infty$ ($P$ is finite and $Q$ blows-up as before);
\item $\mu = \frac{1}{2}, \frac{3}{2}, \frac{5}{2}, \dots$ implies that $A \, P_1^\mu(\tanh(x)) + B \, Q_1^\mu(\tanh(x))$ blows-up exponentially at $-\infty$ ($Q$ is finite and $P$ blows-up as before);
\end{itemize}
so that we can choose from the beginning $\mu \not= \frac{1}{2}, 1, \frac{3}{2}, 2, \frac{5}{2}, 3, \dots$. From the expansion for $y \to 1^-$ we know that:
\[
A \, P_1^\mu(y) + B \, Q_1^\mu(y) \underset{y \to 1^-}{\sim} \left[\frac{A}{\Gamma(1-\mu)} + \frac{B}{2} \cos(\mu \pi) \Gamma(\mu) \right] \left(\frac{2}{1-y}\right)^\frac{\mu}{2},
\]
so we need
\[
A = -\frac{\Gamma(\mu)\Gamma(1-\mu)}{2} \cos(\mu \pi) B.
\]
We now turn to the expansions for $y \to -1^+$. We have that
\[
\begin{split}
Q_1^\mu(y) &= -\cos((1 + \pi)\pi)Q_1^\mu(-y) - \frac{\pi}{2} \sin((1+\mu)\pi)P_1^\mu(-y) \\
           &\underset{y \to -1^+}{\sim} -\frac{1}{2} \cos((1+\mu)\pi)\cos(\mu\pi)\Gamma(\mu) \left(\frac{2}{1+y}\right)^\frac{\mu}{2} + \\
           &\qquad \qquad - \frac{\pi}{2} \sin((1+\mu)\pi) \frac{1}{\Gamma(1-\mu)} \left(\frac{2}{1+y}\right)^\frac{\mu}{2}
\end{split}
\]
and
\[
\begin{split}
P_1^\mu(y) &= -\frac{2}{\pi} \sin((1+\mu)\pi) Q_1^\mu(-y) + \cos((1+\mu)\pi)P_1^\mu(-y) \\
           &\underset{y \to -1^+}{\sim} -\frac{2}{\pi} \sin((1+\mu)\pi) \frac{1}{2} \cos(\mu\pi)\Gamma(\mu)\left(\frac{2}{1+y} \right)^\frac{\mu}{2} + \\
           &\qquad \qquad + \cos((1+\mu)\pi)\frac{1}{\Gamma(1-\mu)} \left(\frac{2}{1+y}\right)^\frac{\mu}{2}.
\end{split}
\]
We hence need to check whether it is possible to have (writing the full expansion for $A P_1^\mu + B Q_1^\mu$ and substituting the value of $A$ we found before)
\[
\begin{split}
-\frac{\Gamma(\mu)\Gamma(1-\mu)}{2} \cos(\mu \pi) \left[-\frac{\Gamma(\mu)}{\pi} \sin((1+\mu)\pi)\cos(\mu\pi) + \frac{\cos((1+\mu)\pi)}{\Gamma(1-\mu)}\right] + \\
 + \left[-\frac{\Gamma(\mu)}{2} \cos((1+\mu)\pi)\cos(\mu\pi) - \frac{\pi}{2} \frac{\sin((1+\mu)\pi)}{\Gamma(1-\mu)}\right] = 0
\end{split}
\]
for some $\mu$. This is the only case, indeed, for which the solution does not grow exponentially as $x \to -\infty$ ($y \to -1^+$). This equation in $\mu$ can be simplified to
\[
-\Gamma(1-\mu)\Gamma(\mu)\cos(\pi\mu)\cot(\pi\mu) \left[\Gamma(\mu)\Gamma(1-\mu)\sin((1+\mu)\pi) + 2\pi \right] = \pi^2
\]
and one can check numerically that it does not exist a $\mu \in \mathbb{R}_{>0}$ that satisfies this last expression.

Hence we get
\[
\tilde{v}_j \not\equiv 0 \iff -\left(\frac{\pi j}{\lambda}\right)^2 = -1 \iff j = \frac{\lambda}{\pi}.
\]

\bibliography{4Dliouville.bib}

\end{document}